\documentclass[11pt,reqno]{amsart}
\usepackage{graphicx,amsmath,amssymb,amsfonts,amsthm,enumitem,bm,xcolor,mathrsfs}
\usepackage{mathtools}
\usepackage{a4wide,fullpage}

\usepackage{mathtools}

\usepackage{amsbsy}

\usepackage[utf8]{inputenc}

\usepackage[normalem]{ulem} 

\usepackage{url}

\usepackage{color}
\usepackage{longtable}
\usepackage{comment}

\newcommand{\N}{\mathbb{N}}
\newcommand{\Z}{\mathbb{Z}}
\newcommand{\R}{\mathbb{R}}
\newcommand{\C}{\mathbb{C}}

\renewcommand{\Re}{\operatorname{Re}}

\renewcommand{\a}{\alpha}
\renewcommand{\b}{\beta}

\renewcommand{\r}{{\rho}}
\renewcommand{\t}{\tau}

\newcommand{\z}{\zeta}

\newcommand{\sgn}{\operatorname{sgn}}

\renewcommand{\(}{\left\(}
\renewcommand{\)}{\right\)}

\numberwithin{equation}{section}

\theoremstyle{plain}
\newtheorem{theorem}{Theorem}[section]
\newtheorem{lemma}[theorem]{Lemma}

\newtheorem{remark}{Remark}
\newtheorem*{remark*}{Remark}

\newtheorem*{remarks*}{Remarks}

\newtheorem*{definition*}{Definition}

\newcommand{\lrb}[1]{\left(#1\right)}

\newtheorem{corollary}[theorem]{Corollary}
\newtheorem{proposition}[theorem]{Proposition}

\newtheorem*{example*}{Example}

\numberwithin{equation}{section}

\renewcommand{\binom}[2]{\left(\begin{smallmatrix}#1\\\\#2\end{smallmatrix}\right)}

\renewcommand{\pmod}[1]{\ \left( \mathrm{mod} \, #1 \right)}
\newcommand{\Pmod}[1]{\ \left( \mathrm{mod} \, #1 \right)}
\newcommand{\Log}{\mathrm{Log}}

\setlist[enumerate]{leftmargin=*,label=\rm{(\arabic*)}}

\makeatletter
\newcommand{\vast}{\bBigg@{3}}
\newcommand{\Vast}{\bBigg@{4}}
\newcommand{\Vvast}{\bBigg@{5}}
\makeatother

\makeatletter
\@namedef{subjclassname@2020}{%
	\textup{2020} Mathematics Subject Classification}
\makeatother

\allowdisplaybreaks

\title{Ramanujan's and Lim's Identities and \\ Harmonic Maass--Jacobi Forms}

\author{Kathrin Bringmann}
\author{Rajat Gupta}
\author{Badri Vishal Pandey}
\address{University of Cologne, Department of Mathematics and Computer Science, Weyertal 86-90, 50931 Cologne, Germany}
\email{kbringma@math.uni-koeln.de}
\email{bpandey@uni-koeln.de, badrivishal9451@zohomail.in}

\address{Department of Mathematics, Indian Institute of Technology Jodhpur, Karwar, Jodhpur,  Rajasthan 342030, India}
\email{rgupta@iitj.ac.in, rajatgpt972@gmail.com}


\begin{document}
\begin{abstract}
	We study an extension of Ramanujan's identities for odd zeta values by Lim and introduce Jacobi analogues of classical Eichler integrals of Eisenstein series. In negative weight we construct explicit completions and embed these objects into a modular framework by showing that they are (singular) harmonic Maass--Jacobi forms. We further describe their non-holomorphic parts in terms of Eichler integrals, establish Ramanujan-type inversion formulas, and study their behavior under the Maass raising and lowering operators and at torsion points.
\end{abstract}
	\maketitle

\section{Introduction and statement of results}
Ramanujan discovered many remarkable identities that have attracted mathematicians and physicists from various areas. Among the striking identities recorded in his papers and lost notebook is a remarkable formula \cite[pp.~319--320, identity~(28)]{R88} for the odd values of the Riemann zeta function $\zeta$.
It states that, for $\alpha,\beta>0$ satisfying $\alpha\beta=\pi^2$ and $a\in\mathbb{Z}$,
\begin{multline}
\alpha^{1-a} \left(\frac{\zeta(2a-1)}{2}+\sum_{n\ge1}\frac{1}{n^{2a-1}(e^{2\alpha n}-1)}\right)
-(-\beta)^{1-a}
\left(\frac{\zeta(2a-1)}{2}+\sum_{n\ge1}\frac{1}{n^{2a-1}(e^{2\beta n}-1)}\right) \\
=2^{2a-2}\sum_{\ell=0}^{a}
\frac{(-1)^{\ell+1} B_{2\ell}B_{2a-2\ell}}
{(2\ell)!(2a-2\ell)!}
\alpha^{a-\ell}\beta^{\ell},
\label{eq:Ramanujan-odd-zeta-value}
\end{multline}
where $B_{\ell}$ is the $\ell$-th Bernoulli number.  
This identity has been reformulated and extended in several directions (see e.g. \cite{BS17}). From a modern perspective, \eqref{eq:Ramanujan-odd-zeta-value} encodes the transformation laws satisfied by the classical Eisenstein series (for $\mathrm{SL}_2(\Z)$) and their associated Eichler integrals. To make this connection precise, let $a\in \mathbb{Z}$ and set $q:=e^{2\pi i \tau}$ with $\tau \in \mathbb{H}$. 
Define
\begin{align}\label{eq:Eichler-Eisenstein}
	\mathcal{E}_{2-a}(\tau) &:= \sum_{n\ge1} \frac{q^n}{n^{a-1}(1-q^n)}=\sum_{n\geq 1}\sigma_{1-a}(n)q^{n},
\end{align}
where $\sigma_{\ell}(n):=\sum_{d\mid n}d^{\ell}$ is the \emph{$\ell$-th divisor function}.
Set $\a=-\pi i \tau$ in \eqref{eq:Ramanujan-odd-zeta-value}. If $a \leq -1$, then \eqref{eq:Ramanujan-odd-zeta-value} gives
\begin{align*}
	G_{2-2a}\lrb{-\frac{1}{\tau}} &= \tau^{2-2a} G_{2-2a}(\tau),
\end{align*}
where $G_k$ is the classical {\it Eisenstein series} (for $\mathrm{SL}_2(\Z)$), defined for even weight $k\geq 2$ by
\begin{align*}
G_{k}(\tau):=-\frac{B_k}{2k}+\mathcal{E}_{k}(\tau).
\end{align*} 
By convention, $G_k(\tau):=0$ for $k$ odd. For $a=0$, \eqref{eq:Ramanujan-odd-zeta-value} becomes the quasimodular inversion transformation formula for $G_2(\tau)$, that is, 
\begin{align*}
G_{2}\left(-\frac{1}{\tau} \right)=\tau^2 G_{2}(\tau)+ \frac{i\t}{4\pi }.
\end{align*}
Finally, if $a\ge1$, then \eqref{eq:Ramanujan-odd-zeta-value} boils down to the modular inversion formula of the Eichler integral of the Eisenstein series $G_{2a}(\tau)$, which is essentially a $(2a-1)$-th anti-derivative of $G_{2a}(\tau)$. In this case $\mathcal{E}_{2-2a}(\tau)$ fails to transform modularly. However, for $a \ge 2$, $\mathcal{E}_{2-2a}(\tau)$ admits a modular completion. More precisely, we call a real-analytic function $\widehat{f}(\tau,\overline{\tau})$ a {\it modular completion of weight $k$} of $f(\tau)$ if $\widehat{f}(\tau,\overline{\tau})$ transforms modularly of weight $k$ and if there exists a constant $c\in\mathbb{C}$ such that, with $\tau=u+iv$, we have
\begin{align*}
	\lim_{\overline{\tau}\to-i\infty}\lrb{\widehat{f}(\tau,\overline{\tau})-cv^{1-k}} = f(\tau).
\end{align*}
Then, (see \cite[Corollary~6.16]{BFOR})
\begin{align*}
\widehat{\mathcal{E}}_{2-2a}(\tau):=\frac{(4\pi v)^{2a-1}B_{2a}}{2(2a)!}+\zeta(2a-1)+\mathcal{E}_{2-2a}(\tau)+\sum_{n\geq 1}\sigma_{1-2a}(n)\Gamma^{*}\left(2a-1,4\pi nv \right)q^{-n}
\end{align*}
is the modular completion of $\smash{\zeta(2a-1)+\mathcal{E}_{2-2a}(\tau)}$ to a harmonic Maass form of weight $\smash{2-2a}$. Here $\Gamma^*$ is the normalized incomplete gamma function defined in \eqref{incompletegamma}. In particular, $\smash{\widehat{\mathcal{E}}_{2-2a}(\tau)}$ is a real-analytic modular form that is annihilated by the weight $k$ {\it hyperbolic Laplacian} $\smash{\Delta_k:=-\xi_{2-k}\circ\xi_k}$, where\footnote{The operator $\xi_{k}$ is called the {\it shadow operator}, which is related to the {\it Maass lowering operator}
$
	\smash{L:=-2i v^2\frac{\partial}{\partial \overline{\tau}}.}
$
The Maass lowering operator decreases the weight of a (weak) Maass form by two and the eigenvalue by $k$ while preserving the modularity.
} $\smash{\xi_k:=-2iv^k\overline{\tfrac{\partial}{\partial\overline{\tau}}}}$. 
Lastly, for $a=1$, it was shown in \cite[Theorem~1.2 (1)]{BOW} that $\mathcal{E}_{2-2a}(\tau)$ has a completion to a sesquiharmonic Maass form which is a real-analytic modular form that is annihilated by $\Delta_{k,2}:=-\xi_k\circ\xi_{2-k}\circ\xi_k$ instead.

In Ramanujan’s lost notebook \cite{bln4}, a large number of identities of the type \eqref{eq:Ramanujan-odd-zeta-value} were recorded. Among these, as well as their later developments, we encounter an interesting result due to Lim \cite[Theorem~2.3]{lim}, which extends a theorem of Ramanujan. Specifically, let $\alpha, \beta > 0$ satisfy $\alpha \beta = \pi^2$, and let $0 < x < 1$ with $a \in \mathbb{Z}$. Then we have that
\begin{align} &\a^{1-a}\sum_{n\ge1}\frac{(-1)^n\sinh(n(1-2x)\a)\sin(\pi n(1-2x))}{n^{2a-1}\sinh(n\alpha)} \nonumber\\ &\hspace{4cm}=-(-\b)^{1-a}\sum_{n\ge1}\frac{(-1)^n\sinh(n(1-2x)\b)\sin(\pi n(1-2x))}{n^{2a-1}\sinh(n\b)} \nonumber\\ &\hspace{6.5cm}-2^{2a-1}\pi\sum_{n=0}^{a-1}\frac{B_{2n+1}(x)B_{2a-1-2n}(x)}{(2n+1)!(2a-1-2n)!}\a^{a-n-1}(-\b)^{n},\label{eq:Lim} 
\end{align}
where $B_n(x)$ is the $n$-th Bernoulli polynomial.

To interpret this identity, we introduce a two-variable analogue of \eqref{eq:Eichler-Eisenstein}.
For $a \in \mathbb{Z}, z = x + iy \in \mathbb{C}$ with $0 \leq y < v$ and $\z:=e^{2\pi iz}$, define
\begin{align}\label{def:P-ell+1}
\mathcal E_{2-a}(z;\tau) := \sum_{n \in \mathbb{Z} \setminus \{0\}} \frac{\z^n}{n^{a-1}\left(1 - q^n\right)}.
\end{align}
The function $z\mapsto\mathcal E_{2-a}(z;\tau)$ is analytic if $0 < y < v$. It has a pole of order $2-a$ at $z=0$ if $a \leq 1$, a logarithmic singularity if $a = 2$, and it is analytic at $z=0$ if $a \geq 3$.

Setting $\alpha = -\pi i \tau$ with $\alpha \beta = \pi^2$, we may rewrite \eqref{eq:Lim} for $a\in \mathbb{Z}$ as 
\begin{multline}\label{th1final}
\sum_{\varepsilon\in \{\pm1\}} \varepsilon\, 
\mathcal{E}_{2-2a}\left(\frac{(1-2x)(1+\varepsilon\tau)+\tau+1}{2};\tau\right) 
\\
= -\tau^{2a-2} \sum_{\varepsilon\in \{\pm1\}} \varepsilon\, 
\mathcal{E}_{2-2a}\left(\frac{(1-2x)\left(1-\varepsilon\frac{1}{\tau}\right)-\frac{1}{\tau}+1}{2};-\frac{1}{\tau}\right)
\\
+ 2(2\pi i)^{2a-1}
\sum_{n=0}^{a-1}
\frac{B_{2n+1}(x)B_{2a-1-2n}(x)}{(2n+1)!(2a-1-2n)!}
\,\tau^{2n}.
\end{multline}

For $a\le1$, the functions $\mathcal{E}_{2-a}(z;\tau)$ are well understood. They appear in connection to vertex operator algebras in the work of Zhu \cite[Section~3]{Z96} on reduction of elliptic $n$-point functions. For $a\le-1$, the function $\mathcal{E}_{2-a}(z;\tau)$ is essentially the $(-a)$-th derivative (with respect to $z$) of the Weierstrass $\wp$-function (see e.g. \cite[Section~2]{BKT}). In particular, $\smash{\mathcal{E}_{2-a}(z;\tau)}$ is a meromorphic Jacobi form of weight $2-a$ and index $0$, that is, it transforms like a Jacobi form (see Section~\ref{subsec:Harmonic Maass--Jacobi forms}) and is allowed to have poles in $z$ for $\tau$ fixed. For $a =0$ or $a=1$, however, $\smash{\mathcal{E}_{2-a}(z;\tau)}$ is a meromorphic quasi-Jacobi form, introduced by Libgober \cite{L}, of weight $2$ and $1$, respectively and index $0$, and hence admits a Jacobi completion. More precisely, we call a function $\smash{\widehat{f}(z;\tau):=\widehat{f}(z,\overline{z};\tau,\overline{\tau})}$ that is real-analytic away from a discrete set of singularities in $(z,\overline{z})$ a {\it Jacobi completion} of $f(z;\tau)$ of weight $k$ and index $m$ if $\smash{\widehat{f}}$ transforms like a Jacobi form of weight $k$ and index $m$ and there exists $\smash{P(X,Y)\in\mathbb{C}[X,Y]}$ such that 
\begin{align*}
	\lim_{\overline{z}\to-i\infty}\lim_{\overline{\tau}\to -i\infty} \lrb{\widehat{f}(z,\overline{z};\tau,\overline{\tau})-P(y,v)} = f(z;\tau).
\end{align*}
For $a \le 1$, define
\begin{align}\label{def:widehat-mathcalE-a-negative}
	\widehat{\mathcal{E}}_{2-a}(z;\tau):=\widehat{\mathcal{E}}_{2-a}(z,\overline{z};\tau,\overline{\tau}) &:= \begin{cases}
												\mathcal{E}_{1}(z;\tau) -\frac{y}{v}+\frac12& \text{if } a=1, \\
												\mathcal{E}_{2}(z;\tau)+\frac{1}{4\pi v}& \text{if } a=0,\\
												\mathcal{E}_{2-a}(z;\tau) & \text{if } a\le -1.
											\end{cases}
\end{align}
Using the relation between $\widehat{\mathcal{E}}_{2-a}(z;\tau)$ and the Weierstrass $\wp$-function (see e.g. \cite[Section~3]{Z96}), it can be shown that $\widehat{\mathcal{E}}_{2-a}(z;\tau)$ transforms like a Jacobi form of weight $2-a$ and index $0$. 

As seen above, for $a\geq 2$, the Eichler integral $\mathcal{E}_{2-2a}(\tau)$ admits a completion to a harmonic Maass form. Thus, it is natural to ask whether $\mathcal{E}_{2-2a}(z;\tau)$ admits a Jacobi completion.
In this paper, we answer this question by constructing an explicit completion. 
More precisely, for $a\ge2$ and $z\in \mathbb{C}$, define the completion of $\mathcal{E}_{2-a}(z;\tau)$ by
\begin{align}
	\widehat{\mathcal{E}}_{2-a}(z;\tau)&:=\mathcal{E}_{2-a}\left(z-\left\lfloor\frac{y}{v}\right\rfloor\tau;\tau\right)-\frac{(-4\pi v)^{a-1}B_{a}\left(\left\{\frac{y}{v}\right\}\right)}{a!} + \mathcal{E}^{-}_{2-a}(z;\tau),\label{def:widehatE-2-a-full}
\end{align}
where the {\it non-holomorphic part} of $\widehat{\mathcal{E}}_{2-a}(z;\tau)$ is given by 
\begin{align}\label{eq:Jacobi-nonholo}
	\mathcal{E}^{-}_{2-a}(z;\tau)&:=\sum_{\substack{m\geq 1 \\\ell\ge1+\left\lfloor\frac{y}{v}\right\rfloor}} \frac{\Gamma^{*}\left(a-1,4\pi  m\left(\ell-\frac{y}{v}\right)v\right)}{m^{a-1}}e^{2\pi im(z-\ell\tau)}\nonumber\\[-1em]
&\hspace{3.4cm}+(-1)^a\sum_{\substack{m\geq 1 \\ \ell\ge-\left\lfloor\frac{y}{v}\right\rfloor}}\frac{\Gamma^{*}\left(a-1,4\pi  m\left(\ell+\frac{y}{v}\right)v\right)}{m^{a-1}}e^{-2\pi im(z+\ell\tau)}.
\end{align}
Here $\lfloor \cdot\rfloor$ denotes the floor function and $\{x\}:= x-\lfloor x\rfloor$ is the fractional part of $x$.


In this paper, we investigate the modularity of \eqref{def:widehatE-2-a-full}. Our first result shows that these completions are harmonic Maass--Jacobi forms, a natural two-variable extension of harmonic Maass forms annihilated by an order-three differential operator, the so-called Casimir operator (see Section~\ref{subsec:Harmonic Maass--Jacobi forms}).
\begin{theorem}\label{thm:mathcalE-2-a-Maass_Jacobi} Assume the notations above.
\begin{enumerate}[leftmargin=*]
	\item\label{point:thm:mathcalE-2-a-Maass_Jacobi:2} For $a\geq 3$, $\widehat{\mathcal{E}}_{2-a}(z;\tau)$ is a harmonic Maass--Jacobi form of weight $2-a$ and index $0$.
		\item\label{point:thm:mathcalE-2-a-Maass_Jacobi:1} For $a=2$, $\widehat{\mathcal{E}}_{0}(z;\tau)$ is a singular harmonic Maass--Jacobi form of weight and index $0$ with logarithmic singularities for $z\in\mathbb{Z}\tau+\mathbb{Z}$.
	\item\label{point:thm:mathcalE-2-a-Maass_Jacobi:3} We have
\begin{align*}
	\lim_{\overline{z}\to-i\infty}\lim_{\overline{\tau}\to -i\infty} \lrb{\widehat{\mathcal{E}}_{2-a}(z;\tau)+\frac{(-4\pi v)^{a-1}B_{a}\left(\left\{\frac{y}{v}\right\}\right)}{a!}} = \mathcal{E}_{2-a}\left(z;\tau\right).
\end{align*}
\end{enumerate}
\end{theorem} 

\begin{remark}
In analogy with classical Eichler integrals, we refer to $\mathcal E_{2-2k}(z;\tau)$ as \textit{Jacobi--Eichler integrals}.
\end{remark}

We next show that the non-holomorphic part of $\widehat{\mathcal{E}}_{2-a}(z;\tau)$ can be represented as an Eichler-type period integral of a suitable modular form, in analogy with the Eichler integral representation \cite[Lemma 5.17]{BFOR} of the non-holomorphic part of harmonic Maass forms. More precisely, define,\footnote{The function $\smash{F_{a}^{[\a,\b]}}$ has a lattice sum representation
\begin{align*}
	F_{a}^{[\a,\b]}(\tau) &= -\frac{(a-1)!}{(2\pi i)^a}\sum_{(m,n)\in\mathbb{Z}^2\setminus\{(0,0)\}} \frac{e^{2\pi i(n\alpha+m\beta)}}{(m\tau+n)^a}.
\end{align*}
In particular, this formula extends $\smash{F_{a}^{[\alpha,\beta]}}$ to all $\alpha\in\mathbb{R}$. Since this is not needed in the present paper, we do not discuss it further.
} for $0<\alpha<1,\beta\in\mathbb{R}$, 
\begin{align}\label{def:F-a-alpha-beta}
	F_{a}^{[\a,\b]}(\tau) 
	&:=\frac{B_a(\alpha)}{a}  - (-1)^{a} \sum_{m,\ell\ge1}  (\ell-\alpha)^{a-1} e^{2\pi im\beta} e^{2\pi im(\ell-\alpha)\tau} \nonumber\\[-0.8em]
	&\hspace{7.5cm} - \sum_{\substack{m\ge1\\ \ell\ge0}}  (\ell+\alpha)^{a-1} e^{-2\pi im\beta} e^{2\pi im(\ell+\alpha)\tau}.
\end{align}	
The function $\smash{F_{a}^{[\a,\b]}}$ is closely related to the theta function with characteristics $g_{\a,\b}$ studied by Zwegers \cite[Section 1.5]{ZW}.
\begin{theorem}\label{thm:non-holomorphic-part-as-integral}
	For $a\ge 2$ and $0<\alpha,\beta<1$,\footnote{It is easy to check that the result actually holds for $\beta\in\mathbb{R}$.} we have 
	\begin{equation*}
		\mathcal{E}^{-}_{2-a}(\alpha\tau+\beta;\tau)= -\frac{(-2\pi)^{a-1}i}{(a-2)!}\int_{-\overline\tau}^{i\infty} \frac{F_{a}^{[\a,\b]}(w)-\frac{B_a(\alpha)}{a} }{ (-i(w+\tau))^{2-a}}dw,
	\end{equation*}
	where $F_{a}^{[\a,\b]}(w)$ is a modular form with characteristics (see Theorem~\ref{thm:F-alpha-beta-transformation}).
\end{theorem}

As an application of Theorem \ref{thm:non-holomorphic-part-as-integral}, we prove a Ramanujan-type formula for $\mathcal{E}_{2-a}(z;\tau)$.

\begin{theorem}\label{thm1.3}
	For $a\ge2$ and $0<\alpha,\beta<1$, we have 
	\begin{align*}
		\mathcal E_{2-a}((1-\beta)\tau+\alpha;\tau) - (-\tau)^{a-2}\mathcal E_{2-a}\left(-\frac{\alpha}{\tau}+\beta;-\frac{1}{\tau}\right) &= (2\pi i)^{a-1}\sum_{n=0}^{a}\frac{(-1)^{n}B_{n}(\beta)B_{a-n}(\alpha)}{n!(a-n)!} \tau^{n-1}.
	\end{align*}
\end{theorem}

\begin{remark}\label{rmk:Lim-general-ale1}
	From the Jacobi transformation of $\widehat{\mathcal{E}}_{2-a}(z;\tau)$ and~\eqref{def:widehat-mathcalE-a-negative} it follows that Theorem~\ref{thm1.3} also holds for $a\le1$. 
\end{remark}



Lim's identity follows directly from Theorem~\ref{thm1.3} and Remark~\ref{rmk:Lim-general-ale1}.

\begin{corollary}\label{cor:thm1.3}
	Equation~\eqref{th1final} holds for $0<x<1$.
\end{corollary}


We also determine the behavior of the completion $\widehat{\mathcal{E}}_{2-a}(z;\tau)$ under the {\it raising operator} of index~$m$
\begin{equation}\label{eq:level-raising}
	Y_m^+ := i \frac{\partial }{\partial z} - 4\pi m \frac{y}{v}.
\end{equation}
This operator maps functions transforming like Jacobi forms of weight $k$ and index $m$ to those of weight $k+1$ and index $m$.
Applying $\smash{Y_0^+}$ to $\widehat{\mathcal{E}}_{2-a}(z;\tau)$, we obtain the following result.
\begin{theorem}\label{thm:action-of-level-raising-operator}
	Let $a\in \mathbb{Z}$. If $ z\notin\Z\tau+\Z$, then we have 
	\[Y_0^+\left(\widehat{\mathcal{E}}_{2-a}(z;\tau)\right) = -2\pi\widehat{\mathcal{E}}_{3-a}(z;\tau).\]
\end{theorem}
Adopting the shorthand notation $``\lim"$ for taking the limit $\lim_{\overline{z}\to-i\infty}\lim_{\overline{\tau}\to -i\infty}$ after subtracting the non-holomorphic constant term from \eqref{def:widehatE-2-a-full}, we have the following commutative diagram: \vspace{0.5cm}
\[
\begin{array}{ccccccccccc}
\cdots 
& \xrightarrow{\;-\frac{1}{2\pi}Y_0^+\;} \hspace{-1cm}
& \widehat{\mathcal{E}}_{-1}(z;\tau) \hspace{-1cm}
& \xrightarrow{\;-\frac{1}{2\pi}Y_0^+\;} \hspace{-1cm}
& \widehat{\mathcal{E}}_{0}(z;\tau) \hspace{-1cm}
& \xrightarrow{\;-\frac{1}{2\pi}Y_0^+\;} \hspace{-1cm}
& \widehat{\mathcal{E}}_{1}(z;\tau) \hspace{-1cm}
& \xrightarrow{\;-\frac{1}{2\pi}Y_0^+\;}
& \cdots 
\\[1.2em]
& & \hspace{1.5cm}\Big\downarrow\lim\;
& & \hspace{1.5cm}\Big\downarrow\lim\;
& & \hspace{1.5cm}\Big\downarrow\lim\;
& &
\\[0.2cm]
\cdots 
& \xrightarrow{\;\frac{1}{2\pi i}\frac{\partial}{\partial z}\;} \hspace{-1cm}
& \mathcal{E}_{-1}(z;\tau) \hspace{-1cm}
& \xrightarrow{\;\frac{1}{2\pi i}\frac{\partial}{\partial z}\;} \hspace{-1cm}
& \mathcal{E}_{0}(z;\tau) \hspace{-1cm}
& \xrightarrow{\;\frac{1}{2\pi i}\frac{\partial}{\partial z}\;} \hspace{-1cm}
& \mathcal{E}_{1}(z;\tau) \hspace{-1cm}
& \xrightarrow{\;\frac{1}{2\pi i}\frac{\partial}{\partial z}\;}
& \cdots .
\end{array}
\vspace{0.5cm}
\]
In particular, the space $\mathbb{C}[\widehat{\mathcal{E}}_{2-a}\colon a\in\mathbb{Z}]$ is closed under the action of $Y_0^+$.

It is well known that a Jacobi form, if specialized to torsion points, yields a modular form \cite[Theorem 1.3]{EZ1985}. Our result shows that, in a similar spirit, $\widehat{\mathcal{E}}_{2-a}(z;\tau)$, if specialized to torsion points, gives rise to a harmonic Maass form.
\begin{theorem}\label{thm:torsion}
	Let $\lambda,\mu\in\mathbb{Q}$, and let $N\in\N$ be minimal such that $N\lambda,N\mu\in\Z$.
	\begin{enumerate}[leftmargin=*]
		\item\label{point:thm:torsion:1} For $a\ge3$, $\widehat{\mathcal{E}}_{2-a}(\lambda\tau+\mu;\tau)$ is a weight $2-a$ harmonic Maass form on $\Gamma(N)$.
		\item\label{point:thm:torsion:2} The function $\widehat{\mathcal{E}}_{0}(\lambda\tau+\mu;\tau)$ is a weight $0$ harmonic Maass form on $\Gamma(N)$ unless $(\lambda,\mu)\in\mathbb{Z}^2$. 
		\item\label{point:thm:torsion:3} For $(\lambda,\mu)\in\mathbb{Z}^2$, the regularized limit $\lim_{z\to0^+}(\widehat{\mathcal E}_0(z;\tau) + 2\log|z|-\log(v))$ is a sesquiharmonic Maass form on $\mathrm{SL}_2(\Z)$.
	\end{enumerate}
\end{theorem}
Let $\smash{D:=q\tfrac{\partial}{\partial q}}$. By \cite[Theorem 5.9]{BFOR}, $D^{a-1}$ maps harmonic Maass forms of weight $2-a$ to modular forms of weight $a$. In view of this and Theorem~\ref{thm:torsion}, it is natural to ask for an explicit  description of $\smash{D^{a-1}(\widehat{\mathcal{E}}_{2-a}(\lambda\tau+\mu;\tau))}$ for $a\geq 2$. As our final result, we explicitly write them in terms of Eisenstein series on $\Gamma(N)$.

\begin{theorem}\label{thm:D-a-1-widehatE-as-Eisenstein-series}
	Assume the notations from Theorem~\ref{thm:torsion} and further assume that $(\lambda,\mu)\notin\mathbb{Z}^2$ if $a=2$. Then, for $a\ge 2$, we have 
	\begin{align*}
		D^{a-1}\left(\widehat{\mathcal{E}}_{2-a}(\lambda\tau+\mu;\tau)\right) &= \frac{ (a-1)!}{(-2\pi i)^a} \sum_{\bm{m}\in\lrb{\Z/N\Z}^2} e^{2\pi i(m_1\mu-m_2\lambda)}G^{[N,\bm{m}]}_{a}(\tau),
	\end{align*}
	where $G^{[N,\bm{m}]}_{a}$ is an Eisenstein series of weight $a$ on $\Gamma(N)$. In particular, for $a\ge3$ and $\lambda=\mu=0$, we have
	$
		D^{a-1}(\widehat{\mathcal{E}}_{2-a}(0;\tau)) = 2G_a(\tau).
	$

\end{theorem}


Note that the Jacobi--Eichler integrals studied in this paper arise naturally in the context of scattering amplitudes in string theory \cite{DHoker, DHokerGreenPioline}. In particular, $\smash{\widehat{\mathcal{E}}_{2-a}(z;\tau)}$ appear implicitly as special cases of Zagier’s single-valued elliptic polylogarithms \cite{Zag}. These, in turn, form the depth-one instances of elliptic modular graph forms \cite{DHokerKleinschmidtSchlotterer, HiddingSchlottererVerbeek}, which are real-analytic modular objects occurring in the low-energy expansion of closed-string scattering amplitudes \cite{DHoker, DHokerGreenPioline}. Representations of elliptic modular graph forms in terms of iterated integrals of $\smash{F^{[\alpha,\beta]}_a(w)}$ given in \eqref{def:F-a-alpha-beta} and their complex conjugates were first developed in \cite{HiddingSchlottererVerbeek} and later studied more systematically in \cite{SchlottererSohnleTao}. Theorem~\ref{thm:non-holomorphic-part-as-integral} may be viewed as the ``depth-one'' instance of this construction. 


The paper is organized as follows. In Section~\ref{sec:prelim}, we recall relevant special functions and review results on the Mellin transform, Eisenstein series, and (sesqui)harmonic Maass (--Jacobi) forms. In Sections~\ref{sec:real-analytic-Eisenstein-series} and \ref{sec:Kronecker--Eisenstein-series}, we introduce non-holomorphic Eisenstein series and study their Maass--Jacobi properties. Section~\ref{sec:Proof of Theorems thm:mathcalE-2-a-Maass_Jacobi and thm:action-of-level-raising-operator} is devoted to the proofs of Theorems~\ref{thm:mathcalE-2-a-Maass_Jacobi} and \ref{thm:action-of-level-raising-operator}. In Sections~\ref{sec:Proof of Theorem thm:non-holomorphic-part-as-integral}, \ref{sec:Proof of Theorem thm1.3 and Corollary cor:thm1.3}, and \ref{sec:Proof of Theorems thm:torsion and thm:D-a-1-widehatE-as-Eisenstein-series}, we prove Theorem~\ref{thm:non-holomorphic-part-as-integral}, Theorem~\ref{thm1.3} together with Corollary~\ref{cor:thm1.3}, and Theorems~\ref{thm:torsion} and \ref{thm:D-a-1-widehatE-as-Eisenstein-series}, respectively. Finally, in Section~\ref{sec:Concluding Remarks}, we discuss directions for future research.

\section*{Acknowledgments}
The authors have received funding from the European Research Council (ERC) under the European Union’s Horizon 2020
research and innovation programme (grant agreement No. 101001179). The second author would also like to acknowledge the Research Initiation Grant (I/RIG/RTG/ 20260015). The authors thank Matthew Krauel, Olav Richter, Oliver Schlotterer, and Sander Zwegers for helpful discussions.

\section{Preliminaries}\label{sec:prelim}

\subsection{Special functions}\label{subsec:special-functions}
Define the {\it normalized incomplete gamma function} by
\begin{align}
\Gamma^*(\a,w):=\frac{\Gamma(\a,w)}{\Gamma(\a)},\quad\text{ with }\quad \Gamma(\a,w):=\int_w^\infty e^{-t} t^{\a-1}dt \label{incompletegamma}
\end{align}  
for $\Re(w)>0$ and $\a \in \mathbb{R}$.
Recall that, for $n\in\N$ and $w\in\mathbb{C}\setminus(-\infty,0]$,
\begin{align}
\Gamma^*(n-1,w) = e^{-w}\sum_{k=0}^{n-2}\frac{w^k}{k!}. \label{incompletegammarep}
\end{align}
We also need 
	\begin{align}\label{eq:der-Gamma}
		\frac{\partial}{\partial w}\lrb{\Gamma(n,w)e^w}&= (n-1)  \Gamma(n-1,w)e^w.
	\end{align}
	Moreover, for $n\in\N$, 
	 \begin{align}\label{eq:Gamma-asymp}
	 	\Gamma(n,x)\sim x^{n}e^{-x}\qquad \text{as } x\to\infty.
	 \end{align}

We also frequently use the \emph{polylogarithm}, defined by 
\begin{align}\label{Li}
\text{Li}_{s}(w):=\sum_{n\ge1} \frac{w^n}{n^s} \qquad (|w|<1).
\end{align}
It admits an analytic continuation to $\sigma:=\Re(s)>0$ and $w\in\C\setminus[1,\infty)$. In particular, $\text{Li}_{s}(w)$ is defined by the series representation \eqref{Li} for $|w|=1$ if $\sigma>2$ since the series is absolutely convergent and is bounded by $\zeta(\sigma)$.
Moreover, for $\sigma\ge1$ and $0\le w<1$, we have
\begin{align}\label{LiHur}
\text{Li}_{s}\left(e^{2\pi i w}\right)+e^{\pi is}\text{Li}_{s}\left(e^{-2\pi i w}\right)=\frac{(2\pi)^s e^{\frac{\pi is}{2}}}{\Gamma(s)}\zeta\left(1-s,w\right),
\end{align}
where $\zeta$ is the \emph{Hurwitz zeta function} defined, for $\sigma>1$ and $\a\notin-\N_0$,
\begin{align}\label{def:Hurwitz-zeta}
\zeta(s,\a):=\sum_{n\ge0}\frac{1}{(n+\a)^s}.
\end{align}
It is well known that $\zeta(s,\alpha)$ has a meromorphic continuation to the complex $s$-plane with the only singularity being a simple pole at $s = 1$ with residue $1$ (see below \cite[25.11.1]{Nist}).  

Moreover, we require the {\it Bernoulli polynomials}, which satisfy
\begin{align*}
\qquad \qquad\qquad \qquad\sum_{n\ge0}B_{n}(x)\frac{t^n}{n!}=\frac{te^{xt}}{e^{t}-1},\qquad \qquad\qquad (\text{for }|t|<2\pi).
\end{align*}
We have 
\begin{align}
B_{n}(1-x)=(-1)^nB_{n}(x),\quad \frac{d}{dx}B_n(x)=nB_{n-1}(x).\label{eq:sym-Ber}
\end{align}
We also recall, for $a\in\N$, 
\begin{align}\label{eq:Hurwitz-zeta-Bernoulli}
\zeta(1-a,x)=-\frac{B_{a}(x)}{a}.
\end{align}

\subsection{Mellin transform}
For $f\colon \R^+\to \C$, its \textit{Mellin transform}, if it exists, is defined by 
	\begin{align*}
		\mathcal{M}_f(s) := \int_0^\infty f(t)t^{s-1} dt.
	\end{align*}
We require the following lemma which follows directly from~\cite[Proposition~3.1.22 (a)]{CS2017}.
\begin{lemma}\label{lem:Mellin-Holo} 
Let $f \colon \mathbb{R}^+ \to \C$ be a continuous function. Suppose that there exist constants $\delta\in\R^+$ and $A \in \R$ such that
		\begin{align*}
		f(t)  = O\left(e^{-\delta t}\right) \quad (t \to \infty) \quad\text{and}\quad
		f(t)  = O\left(t^{-A}\right) \quad \left(t \to 0^+\right).
		\end{align*}
		Then $\mathcal{M}_f(s)$ is holomorphic in the right half-plane $S_A := \{s \in \mathbb{C} \colon \sigma > A\}$. 
\end{lemma}

\subsection{Eisenstein series}\label{subsec:Eisenstein} We recall classical Eisenstein series on $\Gamma(N)$, following \cite[Chapter~4]{DS2005}. For $\smash{\bm{m}=(m_1,m_2)\in(\Z/N\Z)^2}$ such that\footnote{Note that in general we have
$
	G^{[N,\bm{m}]}_{k}= \frac{1}{\ell^k}G^{[\frac{N}{\ell},\frac{\bm{m}}{\ell}]}_{k}
$, where $\ell:=\gcd(m_1,m_2,N)$.} $\gcd(m_1,m_2,N)=1$, define 
\begin{align*}
	G^{[N,\bm{m}]}_{k}(\tau) &:= \sum_{\substack{(c,d)\equiv\bm{m}\Pmod{N}\\(c,d)\neq (0,0)}}\frac{1}{(c\tau+d)^k}. 
\end{align*}
In particular, $\smash{G_k^{[1,(1,1)]}(\tau)=\tfrac{2(2\pi i)^k}{(k-1)!}G_k(\tau)}$. For $k\ge 3$, the functions $G^{[N,\bm{m}]}_{k}$ are modular forms on $\Gamma(N)$ and generate the space of Eisenstein series of weight $k$ on $\Gamma(N)$ (see \cite[Theorem~4.2.3]{DS2005}). As in the case of the classical Eisenstein series, the weight $k=2$ is special. The weight two Eisenstein series $\smash{G^{[N,\bm{m}]}_{2}}$ becomes modular only after adding an additional non-holomorphic term.
More precisely,
\begin{align*}
	\widehat{G}_2^{[N,\bm{m}]}(\tau) &:= G_2^{[N,\bm{m}]}(\tau)-\frac{\pi}{N^2 v}
\end{align*}
is modular of weight two on $\Gamma(N)$.
We also require the Fourier expansion of $G^{[N,\bm{m}]}_{k}(\tau)$ (see \cite[Theorem~4.2.3 and pp.~131--132]{DS2005}). Let $\delta_{S}:=1$ if the statement $S$ is true and $\delta_{S}:=0$ otherwise.
\begin{theorem}\label{thm:Fourier-exp-of-Eisenstein-ser}
	Let $k \ge 2$. Then,
	\begin{align*}
	G^{[N,\bm{m}]}_{k}(\tau)& 
	= \frac{\delta_{N|m_1}}{N^k}\lrb{\z\lrb{k;\frac{m_2}{N}}+(-1)^k\z\lrb{k;1-\frac{m_2}{N}}} + \frac{(-2\pi i)^k}{(k-1)!N^k}\!\!\!\! \sum_{\substack{n,r\in\Z\\ nr\ge1\\n\equiv m_1\!\Pmod{N}}}\!\!\!\!\!\!\sgn(r) \zeta_{N}^{rm_2} r^{k-1} q^{\frac{nr}{N}} .
	\end{align*}
\end{theorem}

We also require the {\it Dedekind eta function}
\begin{align}\label{def:Dedekind-eta-function}
	\eta(\tau):= q^{\frac{1}{24}} \prod_{n\ge 1} \left(1-q^n\right).
\end{align}
It is a weight $\tfrac12$ modular form on $\mathrm{SL}_2(\Z)$ with a multiplier system. The logarithmic derivative of eta is given by $\smash{D(\Log(\eta))=-G_2}$, where $\Log$ denotes the principal branch of logarithm.

\subsection{(Sesqui)harmonic Maass forms}\label{subsec:harminic-Maass-forms} We define and recall basic properties of (sesqui)harmonic Maass forms.
\begin{definition*}\rm
A \emph{harmonic Maass form} of weight $k\in\Z$ on $\Gamma\subset\mathrm{SL}_2(\Z)$ is a smooth function
$f:\mathbb{H}\to\mathbb{C}$ satisfying the following:
\begin{enumerate}
\item For $\gamma=\begin{psmallmatrix}a&b\\ c&d\end{psmallmatrix}\in\Gamma$, we have 
$
f(\gamma\tau)=(c\tau+d)^k f(\tau).
$
\item We have $\Delta_k(f) = 0$.
\item The function $f$ has at most linear exponential growth at the cusps of $\Gamma$.
\end{enumerate}
If condition (2) is replaced with $\xi_{k}\circ\Delta_{k}(f)=0$, then we call $f$ a \emph{sesquiharmonic Maass form}.
\end{definition*}

Finally, we recall Bol's identity (see \cite[Theorem~5.5]{BFOR}).
\begin{lemma}\label{lem:action-od-D-k-on-HarmonicsMaassforms}
	The operator $D^{k-1}$ maps harmonic Maass forms surjectively onto the space of weakly holomorphic modular forms.
\end{lemma}


\subsection{Harmonic Maass--Jacobi forms}\label{subsec:Harmonic Maass--Jacobi forms}
We begin by recalling the definition of classical Jacobi forms.
\begin{definition*}\rm
	Let $k,m\in\mathbb{Z}$. A holomorphic function $\phi:\mathbb{C}\times\mathbb{H}\to\mathbb{C}$ is called a {\it Jacobi form} of weight $k$ and index $m$ if the following hold:
	\begin{enumerate}
		\item For $\gamma=\begin{psmallmatrix} a & b \\ c & d \end{psmallmatrix}\in {\rm SL}_2(\mathbb{Z})$, we have
\begin{align*}
\hspace{0.5cm}\phi\!\left(\frac{z}{c\tau+d};\frac{a\tau+b}{c\tau+d}\right)
&=(c\tau+d)^k e^{\frac{2\pi i mcz^2}{c\tau+d}}\phi(z;\tau).
\end{align*}
	\item For $\ell,n\in\mathbb{Z}$, we have
\begin{align*}
\hspace{0.5cm}\phi\!\left(z+\ell\tau+n;\tau\right)
&=q^{-\ell^2}\z^{-2\ell}\phi(z;\tau).
\end{align*}
	\item For every $\lambda,\mu\in\mathbb{Q}$, we have $\phi(\lambda\tau+\mu;\tau)=O(e^{rv})$ as $v\to\infty$ for some $r>0$.
	\end{enumerate}
\end{definition*}
The classical {\it Jacobi theta function} (see \cite[p.~578]{BBG} or \cite[Proposition 1.3 (8)]{ZW})
\begin{equation}\label{def:Jacobi-theta-function}
	\vartheta(z;\tau) := \sum_{n\in \Z+\frac{1}{2}} e^{2\pi i n \left(z+\frac{1}{2}\right)} q^{\frac{n^2}{2}} = - i q^{\frac{1}{8}} \zeta^{-\frac{1}{2}} \prod_{n\ge 1} \left(1-q^n\right) \left(1-\zeta q^{n-1}\right) \left(1-\zeta^{-1} q^n\right),
\end{equation}
is an example of a Jacobi form of weight and index $\tfrac12$ satisfying:
\begin{align*}
\vartheta(z+1;\tau) &= -\vartheta(z;\tau), 
& \vartheta(z+\tau;\tau) &= -q^{-\frac12}\z^{-1}\vartheta(z;\tau),\\
\vartheta(z;\tau+1) &= e^{\frac{\pi i}{4}}\vartheta(z;\tau),
& \vartheta\lrb{\frac{z}{\tau};-\frac{1}{\tau}} &= -i\sqrt{-i\tau} e^{\frac{\pi i z^2}{\tau}} \vartheta(z;\tau).
\end{align*} 
Specializing the elliptic variable of a Jacobi form to torsion points essentially gives a modular form. More precisely, we have the following result \cite[Theorem~1.3]{EZ1985}.
\begin{lemma}\label{lem:EichleZagier-torsion}
	Let $\lambda,\mu\in\mathbb{Q}$. If $\phi$ transforms like a Jacobi form of weight $k$ and index $m$, then $f(\tau):=e^{2\pi im \lambda^2\tau}\phi(\lambda\tau+\mu;\tau)$ transforms like a modular form of weight $k$ on some congruence subgroup. If $m=0$, then  the congruence subgroup can be taken as $\Gamma(N)$, where $N\in\N$ is minimal such that~$N\lambda,N\mu\in\Z$.
\end{lemma}

Next, we recall harmonic Maass--Jacobi forms (see e.g. \cite{BR}). First, we define the \emph{Casimir operator} $C_{k,m}$, of weight $k$ and index $m$ as
\begin{align*}
	C_{k,m}&:=-4vy\lrb{\frac{\partial}{\partial z}\frac{\partial^2}{\partial \overline{z}^2}+\frac{\partial^2}{\partial z^2}\frac{\partial}{\partial \overline z}}-4v^2\lrb{\frac{\partial}{\partial \overline{\tau}}\frac{\partial^2}{\partial z^2}+\frac{\partial}{\partial \tau}\frac{\partial^2}{\partial \overline{z}^2}} +2ikv\lrb{\frac{\partial}{\partial z}\frac{\partial}{\partial \overline z}+\frac{\partial^2}{\partial \overline{z}^2}} \\
	&\hspace{1cm}+4\pi im \lrb{8v^2\frac{\partial}{\partial \tau}\frac{\partial^2}{\partial \overline{\tau}^2}-2y^2\frac{\partial^2}{\partial \overline{z}^2}+8vy\frac{\partial}{\partial \tau}\frac{\partial}{\partial \overline z}-2i(2k-1)v\frac{\partial}{\partial \overline{\tau}}+2kiy\frac{\partial}{\partial \overline z}}.
\end{align*}
In particular for index $0$, the Casimir operator simplifies as 
\begin{align*}
C_{k,0} &= -4vy\left(\frac{\partial}{\partial z} \frac{\partial^2}{\partial \overline z^2} + \frac{\partial ^2}{\partial z^2} \frac{\partial}{\partial \overline z}\right) -4v^2\left(\frac{\partial}{\partial \overline\tau} \frac{\partial^2}{\partial  z^2} + \frac{\partial}{\partial \tau} \frac{\partial^2}{\partial \overline z^2}\right) + 2ikv\left(\frac{\partial}{\partial z}\frac{\partial}{\partial \overline z} + \frac{\partial ^2}{\partial \overline{z}^2}\right).
\end{align*}

\begin{definition*}\rm
A function $\phi : \mathbb{C} \times \mathbb{H} \to \mathbb{C}$ is called a {\it singular harmonic Maass--Jacobi form} of weight $k \in \mathbb{Z}$ and index $m \in \mathbb{N}_0$ if, for each fixed $\tau \in \mathbb{H}$, the function $z \mapsto \phi(z;\tau)$ is real-analytic on $\mathbb{C}$ away from a discrete set of singularities, and if the following conditions hold:
\begin{enumerate}
\item The function $\phi$ transforms like a Jacobi form.
\item We have 
$
C_{k,m}(\phi)=0.
$
\item For every $\lambda,\mu \in \mathbb{Q}$ such that $\phi(\lambda\tau+\mu;\tau)$ does not have a singularity, we have $ \smash{\phi(\lambda\tau+\mu;\tau)=O(e^{rv})}$ as $v\to\infty$ for some $r>0$.
\end{enumerate}
If $\phi$ is real-analytic on $\mathbb{C}\times\mathbb{H}$, then we call it a \emph{harmonic Maass--Jacobi form}. 
\end{definition*}


\section{A real-analytic Eisenstein series}\label{sec:real-analytic-Eisenstein-series}
Throughout this section, let $a\ge 3$. Define the real-analytic Eisenstein series by
\begin{align}\label{def:mathbbE-2-a}
	\mathbb{E}_{2-a}(z;\tau) &:=v^{a-1} \sum_{(m,n)\in\Z^2\setminus\{(0,0)\}} \frac{ e^{2\pi i\left(mz-\left(m\tau+n\right)\frac{y}{v}\right)}}{(m\tau+n)(m\overline{\tau}+n)^{a-1}}.
	\end{align}
This function arises as a special case of Zagier's single-valued
elliptic polylogarithms \cite{Zag} whose significance for the low-energy expansion of loop-level closed-string amplitudes was recognized in \cite{DHoker,DHokerGreenPioline}.

\subsection{The function $\mathbb{E}_{2-a}(z;\tau)$ as a Maass--Jacobi form}\,
\begin{theorem}\label{thm:mathbbE-Maass--Jacobi}
	The function $\mathbb{E}_{2-a}(z;\tau)$ is a harmonic Maass--Jacobi form of weight $2-a$ and index $0$.
\end{theorem}
\begin{proof}
	It was shown in \cite[pp.~41--49, Section~1.5]{SI} that $\mathbb{E}_{2-a}(z;\tau)$ transforms like a Jacobi form of weight $2-a$ and index $0$. Next, we prove that $\mathbb{E}_{2-a}(z;\tau)$ is annihilated by $C_{2-a,0}$. More precisely, we show that $C_{2-a,0}(g_a(z;\tau))=0$, where
\begin{equation*}
g_a(z;\tau) 
			:= f_{a-1}(z;\tau) (m\tau+n)^{a-2},
	\qquad \text{with} \quad
	f_{a-1}(z;\tau) := v^{a-1}\frac{e^{2\pi i\left(mz-(m\tau+n)\frac{y}{v}\right)} }{\left|m\tau+n\right|^{2a-2}}.
\end{equation*}
	For this, we write
\begin{multline}\label{Ver}
C_{2-a,0}(g_a(z;\tau)) = C_{0,0}\left(f_{a-1}(z;\tau)(m\tau+n)^{a-2}\right)\\ + 2i(2-a)v(m\tau+n)^{a-2} \left(\frac{\partial}{\partial z}\frac{\partial}{\partial \overline z} + \frac{\partial^2}{\partial \overline{z}^2}\right) f_{a-1}(z;\tau).
\end{multline}
Since it follows from the proof of \cite[Proposition 3.2]{BBG} that $f_{a-1}(z;\tau)$ is annihilated by $C_{0,0}$, \eqref{Ver} becomes 
\begin{align}\label{eq:C-2-a-mid-step}
	&C_{2-a,0}(g_a(z;\tau))
	=2i(2-a)v\left(m\tau+n\right)^{a-3} \bigg((m\overline{\tau}+n) \frac{\partial^2}{\partial \overline{z}^2}+(m\tau+n)\frac{\partial}{\partial z} \frac{\partial }{\partial \overline{z}}\bigg) f_{a-1}(z;\tau).
\end{align}
We compute 
\begin{align*}
			\frac{\partial}{\partial z} f_{a-1}(z;\tau) &= -\frac{\pi}{v} \left(m\overline{\tau}+n\right) f_{{a-1}}(z;\tau),\qquad
			\frac{\partial}{\partial \overline{z}} f_{a-1}(z;\tau) =\frac{\pi}{v}(m\tau+n)f_{a-1}(z;\tau).
	\end{align*}
This gives
		\begin{align*}
			\frac{\partial}{\partial z} \frac{\partial}{\partial \overline{z}} f_{a-1}(z;\tau) &= -\frac{\pi^2}{v^2} |m\tau+n|^2 f_{a-1}(z;\tau) ,\qquad
			\frac{\partial^2}{\partial\overline z^2} f_{a-1}(z;\tau)= \frac{\pi^2}{v^2}(m\tau+n)^2f_{a-1}(z;\tau).
		\end{align*}
		Plugging these into \eqref{eq:C-2-a-mid-step} gives that $C_{2-a,0}(g_a(z;\tau))=0$. 
		
		Finally, it is not difficult to show that, for $\alpha,\beta\in\mathbb{Q}$,
		\begin{align}\label{eq:asym-for-age3}
			\mathbb{E}_{2-a}(\alpha\tau+\beta;\tau)=O\left(v^{a-1}\right),
		\end{align}
		which completes the proof.
\end{proof}

\subsection{The relation between $\mathbb{E}_{2-a}(z;\tau)$ and $\widehat{\mathcal{E}}_{2-a}(z;\tau)$} We relate $\mathbb{E}_{2-a}(z;\tau)$ and $\widehat{\mathcal{E}}_{2-a}(z;\tau)$ via their Fourier expansion. We therefore determine the Fourier expansion of $\mathbb{E}_{2-a}(z;\tau)$. 
First define, for $a\ge3$ and $-v<y<v$, 
\begin{align}\label{def:G}
G(z;\tau):=v^{a-1}\sum_{m\geq 1}e^{2\pi im\left(z-\tau\frac{y}{v}\right)} \sum_{n\in\Z} \frac{e^{-2\pi in\frac{y}{v}}}{(m\tau+n)(m\overline\tau+n)^{a-1}}.
\end{align}
Splitting off $m=0$ in \eqref{def:mathbbE-2-a} and using \eqref{Li} and \eqref{LiHur}, we now write $\mathbb{E}_{2-a}(z;\tau)$ in terms of $G(z;\tau)$ and the Hurwitz zeta function.
\begin{lemma}\label{lem:mathbbE-as-G}
	For $a\ge3$ and $0\le y<v$, we have 
	\begin{align*}
\mathbb E_{2-a}(z;\tau)=\frac{(-2\pi i)^a }{(a-1)!}\zeta\left(1-a,\frac{y}{v}\right)v^{a-1}+G(z;\tau)+(-1)^{a}G(-z;\tau).
	\end{align*}
\end{lemma}

Next, we determine the Fourier expansion of $G(z;\tau)$ for $-v < y < v$. Combined with Lemma~\ref{lem:mathbbE-as-G}, this then yields the Fourier expansion of $\mathbb{E}_{2-a}(z;\tau)$ for $0\le y<1$.
\begin{lemma}\label{lem:G_z-tau}
For $a\ge3$ and $-v<y<v$, we have 
	\begin{multline*}
		G(z;\tau)=\frac{\pi i^{a-2}}{ 2^{a-2}} \sum_{m\geq 1}\frac{e^{2\pi im\left(z-\frac{y}{v}\tau\right)}}{m^{a-1}}\vast(\sum_{\ell\ge1} e^{2\pi im\left(\ell+\frac{y}{v}\right)\tau} + \delta_{0<\frac{y}{v}<1}e^{2\pi im\frac{y}{v}\tau} +\delta_{\frac{y}{v}=0}
		 \\ \hspace{2cm}+\sum_{\ell\ge1} \Gamma^*\left(a-1,4\pi m\left(\ell-\frac{y}{v}\right)v\right)  e^{-2\pi i m\left(\ell-\frac{y}{v}\right)\tau} + \delta_{-1<\frac{y}{v}<0}\Gamma^*(a-1,-4\pi my) e^{2\pi im\frac{y}{v}\tau}\vast).
	\end{multline*}
\end{lemma}
\begin{proof}
	We start by writing \eqref{def:G} as 
	\begin{align}\label{def:G-as-sum-F}
		G(z;\tau)=v^{a-1}\sum_{m\geq 1}e^{2\pi im\left(z-\tau\frac{y}{v}\right)} G_{m}^{[a]}(z;\tau),
	\end{align}
	where
\begin{align*}
	G_{m}^{[a]}(z;\tau) &:= \sum_{n\in\Z} \frac{e^{-2\pi in\frac{y}{v}}}{(m\tau+n)(m\overline\tau+n)^{a-1}}.
\end{align*}
Using Theorem~1 of \cite{PW2001} with $\alpha=1,\beta=a-1, \mu=\tfrac{y}{v},$ and $\tau\mapsto m\tau$ gives 
	\begin{align}\label{eq:F_a-Forier-mid}
		G_{m}^{[a]}(z;\tau) &= -\frac{(2\pi i)^a}{(a-2)!}\sum_{\ell\in\Z} c_{m,v,\frac{y}{v}}(\ell) e^{2\pi im\lrb{\ell+\frac{y}{v}}u},
	\end{align}
	where 
	\begin{align}\label{eq:c_n+mu}
		&c_{m,v,\frac{y}{v}}(\ell) :=\begin{cases}
											(a-2)!(4\pi mv)^{1-a} & \text{if }\ell+\frac{y}{v}=0, \\
											\lrb{\ell+\frac{y}{v}}^{a-1}\sigma_{1,a-1}\lrb{4\pi m\lrb{\ell+\frac{y}{v}}v}e^{-2\pi m\lrb{\ell+\frac{y}{v}}v} & \text{if }\ell+\frac{y}{v} >0,\\
											(-1)^{a+1}\lrb{\ell+\frac{y}{v}}^{a-1}\sigma_{a-1,1}\lrb{-4\pi m\lrb{\ell+\frac{y}{v}}v}e^{2\pi m\lrb{\ell+\frac{y}{v}}v} & \text{if }\ell+\frac{y}{v} <0.
										\end{cases}
	\end{align}
	Here, for $s_1\in\mathbb{C},\Re(s_2)$, and $\Re(w)>0$, we have
	\begin{align*}
		\sigma_{s_1,s_2}(w)&:=\int_0^\infty (t+1)^{s_1-1}t^{s_2-1} e^{-wt} dt.
	\end{align*}
	We compute 
	\begin{align*}
		\sigma_{1,a-1}(w)&=\frac{(a-2)!}{w^{a-1}},\qquad
		\sigma_{a-1,1}(w)=(a-2)! \frac{e^w \Gamma^*(a-1,w)}{w^{a-1} }.
	\end{align*}
	Substituting these into \eqref{eq:c_n+mu}, then into \eqref{eq:F_a-Forier-mid} (using that $-1 < \tfrac{y}{v} < 1$) and finally into \eqref{def:G-as-sum-F} we obtain the lemma.
\end{proof}

We now relate $\widehat{\mathcal{E}}_{2-a}(z;\tau)$ to $\mathbb{E}_{2-a}(z;\tau)$.

\begin{theorem}\label{thm:mathbbE-as-mathcalE}
	For $a\ge 3$ and $0\le y<v$, we have 
		\begin{equation*}
			\widehat{\mathcal{E}}_{2-a}(z;\tau) = \frac{(-2i)^{a-2}}{\pi}\mathbb{E}_{2-a}(z;\tau).
		\end{equation*}
\end{theorem}
\begin{proof}
We assume that $y>0$, as the case $y=0$ is similar. Using Lemma~\ref{lem:G_z-tau}, we obtain  
\begin{align}
G(z;\tau)
&=\frac{\pi i^{a-2}}{ 2^{a-2}} \sum_{m\ge1} \frac{e^{2\pi im\left(z-\frac{y}{v}\tau\right)}}{m^{a-1}} \vast(\sum_{\ell\ge0} e^{2\pi im\left(\ell+\frac{y}{v}\right)\tau} \nonumber\\[-1.5em]
&\hspace{6.6cm}+ \sum_{\ell\ge1} \Gamma^*\left(a-1,4\pi m\left(\ell-\frac{y}{v}\right)v\right) e^{-2\pi im\left(\ell-\frac{y}{v}\right)\tau} \vast)\nonumber\\
&\!\!\!\!\!=\frac{\pi i^{a-2}}{ 2^{a-2}}\vast( \sum_{m\geq 1}\frac{e^{2\pi imz}}{m^{a-1}(1-e^{2\pi im \tau})} + \sum_{m,\ell\ge1}\frac{\Gamma^{*}\left(a-1,4\pi  m\left(\ell-\frac{y}{v}\right)v\right)}{m^{a-1}}e^{2\pi im(z-\ell\tau)}\vast).\label{eq:G+z}
\end{align}
Similarly, using Lemma~\ref{lem:G_z-tau}, we get 
\begin{multline*}
G(-z;\tau)
=\frac{\pi i^{a-2}}{ 2^{a-2}}\vast(- \sum_{m\ge1} \frac{e^{-2\pi imz}}{m^{a-1}(1-e^{-2\pi im\tau})}  + \sum_{\substack{m\ge1\\ \ell\ge0}} \frac{\Gamma^{*}\!\!\left(a-1,4\pi  m \left(\ell+\frac{y}{v}\right)v\right)}{m^{a-1}} e^{-2\pi im(z+\ell\tau)}\vast).
\end{multline*}
Inserting this together with \eqref{eq:G+z} into Lemma~\ref{lem:mathbbE-as-G}, using \eqref{eq:Hurwitz-zeta-Bernoulli} and \eqref{def:widehatE-2-a-full} gives the theorem.
\end{proof}

\section{Kronecker--Eisenstein series}\label{sec:Kronecker--Eisenstein-series}
The {\it Kronecker--Eisenstein series} is defined by
\begin{align*}
	\mathbb{E}_0(s;z;\tau) &:= v^s \sum_{(m,n)\in\Z^2\setminus\{(0,0)\}} \frac{e^{2\pi i\left(mz-\left(m\tau+n\right)\frac{y}{v}\right)}}{|m\tau+n|^{2s}}.
\end{align*}
This function is absolutely convergent for $\sigma>1$ and admits an analytic continuation to $\sigma>\tfrac12$ for $\tau\in\mathbb{H}$ fixed and $z\notin\Z\tau+\Z$ (see \cite[p~28, Theorem~2]{SI}). We have the following lemma.

\begin{lemma}\label{lem:Maass-Jacobi-a=2}
	The function $\mathbb{E}_0(1;z;\tau)$ is a singular harmonic Maass--Jacobi form of weight and index $0$. Moreover, $z\in\Z\tau+\Z$ is a logarithmic singularity of $\mathbb{E}_{0}(1;z;\tau)$, in the sense that $\mathbb{E}_{0}(1;z;\tau)\sim -2\pi\log|z-m\tau-n|$ as $z\to m\tau+n$.
\end{lemma}
\begin{proof}
	It was shown in \cite[Sections~1.3, 1.4]{SI} that $\mathbb{E}_0(1;z;\tau)$ transforms like a Jacobi form of weight and index $0$. Moreover, by \cite[Proposition~3.2]{BBG}, it is annihilated by $C_{0,0}$.
We next determine its growth as $v\to\infty$. Recall Kronecker's second limit formula \cite[equation~(39)]{SI} 
\begin{align}\label{eq:Kronecker-second-limit-formula}
\mathbb{E}_0(1;z;\tau)
= -2\pi\log\left|\frac{\vartheta(z;\tau)}{\eta(\tau)}\right| + \frac{2\pi^2 y^2}{v}.
\end{align}
As $v\to\infty$, using \eqref{def:Dedekind-eta-function} and \eqref{def:Jacobi-theta-function}, we have, for $\alpha,\beta\in\mathbb{Q}$ such that $\alpha\tau+\beta\notin\Z\tau+\Z$,
\begin{align}\label{eq:asym-for-a=2}
	\mathbb{E}_0(1;\alpha\tau+\beta;\tau) =O(v),
\end{align}
which satisfies the required growth condition for a singular harmonic Maass--Jacobi form.
Finally, we determine the singularities of $\mathbb{E}_0(1;z;\tau)$.
From \eqref{eq:Kronecker-second-limit-formula} we obtain 
\begin{align*}
\qquad\mathbb{E}_0(1;z;\tau)&\sim -2\pi\log|z|\qquad \text{ as }z\to 0.
\end{align*}
The claimed behavior of $\mathbb{E}_0(1;z;\tau)$ as $z\to m\tau+n$ follows by the ellipticity of $\mathbb{E}_0(1;z;\tau)$ which is a direct consequence of \eqref{def:widehatE-2-a-full}.
\end{proof}

Now we relate $\widehat{\mathcal{E}}_0(z;\tau)$ to $\mathbb{E}_0(z;\tau)$.
\begin{theorem}\label{thm:a=2case}
	For $0\le y<v$ with $z\neq0$, we have 
	\[
	\widehat{\mathcal{E}}_{0}(z;\tau)=\frac{\mathbb{E}_{0}(1;z;\tau)}{\pi}.
	\]
\end{theorem}
\begin{proof}
We plug $a=2$ into \eqref{def:P-ell+1}, to obtain 
\begin{align*}
	\mathcal E_0(z;\tau) &=  -\frac{\pi i}{2} - \pi i z + \frac{\pi i \tau}{6} - \Log\left(\vartheta(z;\tau)\right) + \Log\left(\eta(\tau)\right),
\end{align*}
using \eqref{def:Jacobi-theta-function}.
From this, we conclude
\begin{align*}
	\mathcal E_0(z;\tau) + \overline{\mathcal E_0(z;\tau)} 
	&= 2\pi y - \frac{\pi v}{3} - 2\log\left|\frac{\vartheta(z;\tau)}{\eta(\tau)}\right|. 
\end{align*}
	Plugging this into \eqref{eq:Kronecker-second-limit-formula} gives 
	\begin{align}
		\mathbb E_0(1;z;\tau)= -2\pi^2 y+\frac{\pi^2 v}{3}+\frac{2\pi^2 y^2}{v}+\pi\mathcal{E}_{0}(z;\tau)+\pi\overline{\mathcal{E}_{0}(z;\tau)}.\label{eq:mathbbE-at-a=2-as-mathcalE}
	\end{align}
	Using \eqref{def:widehatE-2-a-full}, proving the theorem is thus equivalent to showing that 
	\begin{align*}
		-2\pi y+\frac{\pi v}{3}+\frac{2\pi y^2}{v}+\overline{\mathcal{E}_{0}(z;\tau)} = 2\pi vB_{2}\left(\frac{y}{v}\right) + \mathcal{E}^{-}_{0}(z;\tau).
	\end{align*}
	This follows by comparing both sides after plugging \eqref{def:P-ell+1} with $a=2$ into the left-hand side and applying \eqref{def:widehat-mathcalE-a-negative} together with \eqref{incompletegammarep} for $n=2$ to the right-hand side.
\end{proof}

\section{Proof of Theorems~\ref{thm:mathcalE-2-a-Maass_Jacobi} and \ref{thm:action-of-level-raising-operator}}\label{sec:Proof of Theorems thm:mathcalE-2-a-Maass_Jacobi and thm:action-of-level-raising-operator} 

\subsection{Proof of Theorem~\ref{thm:mathcalE-2-a-Maass_Jacobi}}
\,
\begin{proof}[Proof of Theorem~\ref{thm:mathcalE-2-a-Maass_Jacobi}]\, \\
	\noindent\ref{point:thm:mathcalE-2-a-Maass_Jacobi:2} By \eqref{def:widehatE-2-a-full}, $\smash{\widehat{\mathcal{E}}_{2-a}(z;\tau)}$ is elliptic. Combining this with Theorems~\ref{thm:mathbbE-as-mathcalE}, and~\ref{thm:mathbbE-Maass--Jacobi}, we conclude that $\smash{\widehat{\mathcal{E}}_{2-a}(z;\tau)}$ is a harmonic Maass--Jacobi form of weight $2-a$ and index $0$.\\	 
	 \noindent\ref{point:thm:mathcalE-2-a-Maass_Jacobi:1} By Theorem~\ref{thm:a=2case}, Lemma~\ref{lem:Maass-Jacobi-a=2}, and \eqref{def:widehatE-2-a-full}, $\smash{\widehat{\mathcal{E}}_0(z;\tau)}$ is a singular Maass--Jacobi form of weight and index $0$ with logarithmic singularities for $z \in \mathbb{Z}\tau + \mathbb{Z}$.\\	 
\noindent\ref{point:thm:mathcalE-2-a-Maass_Jacobi:3} 
	As in \ref{point:thm:mathcalE-2-a-Maass_Jacobi:2} and~\ref{point:thm:mathcalE-2-a-Maass_Jacobi:1}, we may assume that $0 \le y < v$ (with the extra condition $z\neq0$ if $a=2$). 
	 We consider each summand in the non-holomorphic part of $\smash{\widehat{\mathcal{E}}_{2-a}(z;\tau)}$, given in \eqref{eq:Jacobi-nonholo}. We first take the limit $\smash{\overline{\tau}\to- i\infty}$ and then $\smash{\overline{z}\to-i\infty}$ while keeping $\tau$ and $z$ fixed. We begin with the terms corresponding to $\ell \ge 1$. Then
$\smash{4\pi mv(\ell\pm \frac{y}{v}) \sim 4\pi m \ell v}$ as $\overline{\tau}\to-i\infty$,
which implies, using \eqref{eq:Gamma-asymp},
\begin{align}
\frac{\Gamma^{*}\!\left(a-1,4\pi m\left(\ell\pm \frac{y}{v}\right)v\right)}{m^{a-1}}e^{\mp2\pi i m(z\pm \ell\tau)}
&= O_{z, \tau}\lrb{\overline{\tau}^{a-1}e^{-2\pi im \ell \overline{\tau}}}.\nonumber
\end{align}
Plugging this into \eqref{eq:Jacobi-nonholo}, only the contribution from $\ell=0$ survives and we obtain
\begin{align*}
	\lim_{\overline{\tau}\to-i\infty} \widehat{\mathcal{E}}_{2-a}^-(z;\tau) = (-1)^a\sum_{m\ge1} \frac{\Gamma^*\lrb{a-1,4\pi my}}{m^{a-1}}e^{-2\pi imz} \to 0,
\end{align*}
as $\overline{z}\to-i\infty$, which, in view of \eqref{def:widehatE-2-a-full}, is equivalent to the statement of the theorem.
\end{proof}

\subsection{Proof of Theorem~\ref{thm:action-of-level-raising-operator}}

\begin{proof}[Proof of Theorem~\ref{thm:action-of-level-raising-operator}]
By Theorem~\ref{thm:mathcalE-2-a-Maass_Jacobi}~\ref{point:thm:mathcalE-2-a-Maass_Jacobi:2},\ref{point:thm:mathcalE-2-a-Maass_Jacobi:1} if $a\ge2$ and the discussion below \eqref{def:widehat-mathcalE-a-negative} if $a\le1$, $\smash{\widehat{\mathcal{E}}_{2-a}(z;\tau)}$ is elliptic and we may restrict to $0\le y<v$. Again, by Theorem~\ref{thm:mathcalE-2-a-Maass_Jacobi}~\ref{point:thm:mathcalE-2-a-Maass_Jacobi:2}, $\smash{\widehat{\mathcal{E}}_{2-a}(z;\tau)}$ is real-analytic if $a\ge3$. Further, by Theorem~\ref{thm:mathcalE-2-a-Maass_Jacobi}~\ref{point:thm:mathcalE-2-a-Maass_Jacobi:1} if $a=2$ and \eqref{def:widehat-mathcalE-a-negative} together with the discussion below \eqref{def:P-ell+1} if $a\le1$, $\smash{\widehat{\mathcal{E}}_{2-a}(z;\tau)}$ is real-analytic if $z\notin\mathbb{Z}$. Hence, for all $a\in\Z$, we may apply $Y_0^+$ away from $\mathbb{Z}$. First, we determine the action of $Y_0^+$ on the individual term of $\smash{\widehat{\mathcal{E}}_{2-a}(z;\tau)}$ for $a\ge2$ given in \eqref{def:widehatE-2-a-full}.
	From the definition \eqref{def:P-ell+1}, we have, for $a\in\mathbb{Z}$,
	\begin{align}\label{eq:der-mathcalE}
		\frac{1}{2\pi i} \frac{\partial }{\partial z} \mathcal{E}_{2-a}(z;\tau) = \mathcal{E}_{3-a}(z;\tau).
	\end{align}
	Next, using \eqref{eq:sym-Ber}, we have, for $a\in\N$, 
	\begin{align}\label{eq:der-BernoulliPoly}
		\frac{1}{2\pi i} \frac{\partial}{\partial z}\frac{-(-4\pi v)^{a-1}B_{a}\left(\frac{y}{v} \right)}{a!}&= -\frac{(-4\pi v)^{a-2}B_{a-1}\left(\frac{y}{v} \right)}{(a-1)!}.
	\end{align}
	Using \eqref{eq:der-Gamma}, we obtain 
	\begin{align}\label{eq:der-Gamma-2}
	\hspace{-0.2cm}\frac{1}{2\pi i} \frac{\partial}{\partial z} \Gamma^*\lrb{a-1, 4\pi m\lrb{\ell\mp \frac{y}{v}}v}e^{2\pi im (\pm z-\ell\tau)} &= \pm m \Gamma^*\lrb{a-2,4\pi m\lrb{\ell\mp\frac{y}{v}}v}e^{2\pi im(\pm z-\ell\tau)}.
	\end{align}
	Differentiating \eqref{def:widehatE-2-a-full} with respect to $z$, using \eqref{eq:der-mathcalE}, \eqref{eq:der-BernoulliPoly}, and \eqref{eq:der-Gamma-2} and noting that, by \eqref{eq:level-raising}, $Y_0^+=i\tfrac{\partial}{\partial z}$ gives the theorem for
	$a\ge3$. The case $a=2$ follows from Theorem~\ref{thm:a=2case}, \eqref{eq:mathbbE-at-a=2-as-mathcalE}, and \eqref{def:widehat-mathcalE-a-negative}.
Using \eqref{def:widehat-mathcalE-a-negative} and \eqref{eq:der-mathcalE} then yields the case $a\le1$, completing the proof.
\end{proof}

\section{Proof of Theorem~\ref{thm:non-holomorphic-part-as-integral}}\label{sec:Proof of Theorem thm:non-holomorphic-part-as-integral}
First we write $\widehat{\mathcal{E}}_{2-a}^-(z;\tau)$ as a non-holomorphic Eichler integral. For this, define
\begin{align}\label{def:F-a-0-alpha-beta}
	F_{a,0}^{[\alpha,\beta]}(w):=F_a^{[\alpha,\beta]}(w)-\frac{B_a(\alpha)}{a} .
\end{align}
\begin{proposition}\label{prop:non-holomorphic-part-as-integral-2}
	For $a\ge 2$ and $0<\alpha,\beta<1$, we have 
	\begin{equation*}
		\mathcal{F}_{2-a}^{[\alpha,\beta]}(\tau):=\mathcal{E}^{-}_{2-a}(\alpha\tau+\beta;\tau)= -\frac{(-2\pi)^{a-1}i}{(a-2)!}\int_{-\overline\tau}^{i\infty} \frac{F_{a,0}^{[\alpha,\beta]}(w)}{ (-i(w+\tau))^{2-a}}dw.
	\end{equation*}
\end{proposition}
\begin{proof} 
Using \eqref{incompletegamma}, we obtain
\begin{align*}
	\Gamma(a-1,4\pi m(\ell\mp\alpha)v) 
	&= (-2\pi im(\ell\mp\alpha))^{a-1} e^{2\pi im(\ell\mp\alpha)\tau} \int_{-\overline\tau}^{i\infty} e^{2\pi im(\ell\mp\alpha)w}(w+\tau)^{a-2}dw.
\end{align*}

Substituting this into \eqref{eq:Jacobi-nonholo}, interchanging the sum and integral, and recalling definitions \eqref{def:F-a-alpha-beta} and \eqref{def:F-a-0-alpha-beta} gives the proposition.
\end{proof}

Next, we determine the transformation properties of $F_{a}^{[\a,\b]}$. 

\begin{theorem}\label{thm:F-alpha-beta-transformation}
	Let $a\ge2$ and $0<\alpha,\beta<1$.	
\begin{enumerate}[leftmargin=*]
\item\label{eq:F-alpha-beta-translation} We have 
\begin{align*}
	F_{a}^{[\a,\b]}(\tau+1)&=F_{a}^{[\a,\b+\a]}(\tau).
\end{align*}
\item\label{eq:F-alpha-beta-inversion} We have
\begin{align*}
F_{a}^{[\a,\b]}\left(-\frac{1}{\tau}\right) = \tau^a F_{a}^{[1-\b,\a]}(\tau).
\end{align*}
\item\label{eq:F-alpha-beta-elliptic-trans} We have 
\begin{align*}
	F_{a}^{[1- \alpha,1-\beta]}(\tau)= (-1)^{a} F_{a}^{[\alpha,\beta]}(\tau) .
\end{align*}
\end{enumerate}
\end{theorem}
\begin{proof}
First, define, using \eqref{def:widehatE-2-a-full}
and Proposition~\ref{prop:non-holomorphic-part-as-integral-2}
\begin{equation}\label{eq:widehatmathcalE-alpha-beta}
	\widehat {\mathcal E}_{2-a}^{[\alpha,\beta]}(\tau) := \widehat {\mathcal E}_{2-a}(\alpha\tau+\beta;\tau) = \mathcal E_{2-a}(\alpha\tau+\beta;\tau)-\frac{(-4\pi v)^{a-1}B_a\left(\alpha\right)}{a!} +\mathcal F_{2-a}^{[\a,\b]}(\tau).
\end{equation}
Set 
\begin{equation}\label{low}
	f_a^{[\alpha,\beta]}(\tau) := L\!\left(\widehat{\mathcal E}_{2-a}^{[\alpha,\beta]}(\tau)\right) = -\frac{(-4\pi)^{a-1} v^a}{(a-2)!} F_{a}^{[\a,\b]}(-\overline\tau).
\end{equation}


\noindent\ref{eq:F-alpha-beta-translation}
From the ellipticity of $\widehat{\mathcal{E}}_{2-a}(z;\tau)$ (see Theorem~\ref{thm:mathcalE-2-a-Maass_Jacobi} \ref{point:thm:mathcalE-2-a-Maass_Jacobi:2}, \ref{point:thm:mathcalE-2-a-Maass_Jacobi:1}), \eqref{eq:widehatmathcalE-alpha-beta} and \eqref{low} we obtain \ref{eq:F-alpha-beta-translation}.

\noindent\ref{eq:F-alpha-beta-inversion} Using Theorem~\ref{thm:mathcalE-2-a-Maass_Jacobi} \ref{point:thm:mathcalE-2-a-Maass_Jacobi:2},\ref{point:thm:mathcalE-2-a-Maass_Jacobi:1}, \eqref{eq:widehatmathcalE-alpha-beta} and \eqref{low} yields 
\begin{align*}
	f_a^{[\alpha,\beta]}\left(-\frac{1}{\tau}\right) &= (-\tau)^{-a} f_a^{[1-\beta,\alpha]}(\tau).
\end{align*} 
This together with \eqref{low} gives the claim.\\
\noindent\ref{eq:F-alpha-beta-elliptic-trans} Using Theorem~\ref{thm:mathcalE-2-a-Maass_Jacobi} \ref{point:thm:mathcalE-2-a-Maass_Jacobi:2}, \ref{point:thm:mathcalE-2-a-Maass_Jacobi:1}, together with \eqref{eq:widehatmathcalE-alpha-beta} and \eqref{low}, we obtain the claim.
\end{proof}

\begin{proof}[Proof of Theorem~\ref{thm:non-holomorphic-part-as-integral}]
The theorem is an immediate consequence of Proposition~\ref{prop:non-holomorphic-part-as-integral-2}
and Theorem~\ref{thm:F-alpha-beta-transformation}.
\end{proof}


\section{Proof of Theorem~\ref{thm1.3} and Corollary~\ref{cor:thm1.3}}\label{sec:Proof of Theorem thm1.3 and Corollary cor:thm1.3}
First, we determine a functional equation satisfied by the $L$-function associated to\footnote{These $L$-functions also appear in the study of so-called modular regulators and multiple Eisenstein values (see e.g. \cite{BZ23}).} $F_a^{[\alpha,\beta]}$. Throughout this section we assume $a\ge2$.
Define\footnote{We subtract the constant term from $F_a^{[\alpha,\beta]}$ to ensure absolute convergence.} 
\begin{align}\label{def:Lambda-F-alpha-beta}
	\Lambda\lrb{F_a^{[\alpha,\beta]};s}&:= \mathcal{M}_{F_{a,0}^{[\alpha,\beta]}(it)}(s) = \int_0^\infty F_{a,0}^{[\alpha,\beta]}(it)t^{s-1}dt.
\end{align}

\begin{proposition}\label{prop:Lamda-well-define-S_a}
	The function $s \mapsto \Lambda(F_a^{[\alpha,\beta]};s)$ defines a holomorphic function in $S_a$.
\end{proposition}
\begin{proof}
	By \eqref{def:F-a-alpha-beta} and \eqref{def:F-a-0-alpha-beta}, $\smash{F_{a,0}^{[\alpha,\beta]}(it)}$ decays exponentially as $t\to\infty$. Furthermore, by \eqref{def:F-a-0-alpha-beta} and Theorem~\ref{thm:F-alpha-beta-transformation} \ref{eq:F-alpha-beta-inversion}, we have  
	\begin{align}
		F_{a,0}^{[\alpha,\beta]}(it) &= O(t^{-a})\qquad \text{as }t\to 0^+.\label{eq:F-a-converge-near-0}
	\end{align}
	Hence, Lemma~\ref{lem:Mellin-Holo} with $A=a$ gives the proposition.
\end{proof}

Next we prove a functional equation for $\Lambda(F_a^{[\alpha,\beta]};s)$ which yields its meromorphic continuation to the complex $s$-plane. The proof is analogous to \cite[Theorem~11.2.2]{CS2017}. 
\begin{theorem}\label{thm:Lambda-analytic-cont}
	Let $0<\alpha,\beta<1$. The function
	\[
	\Lambda\lrb{F_a^{[\alpha,\beta]};s}+ \frac{1}{a}\!\lrb{\!\frac{B_a(\alpha)}{s} - \frac{ i^{a}B_a(1-\beta)}{s-a} \!}\!
	\]
	is holomorphic in the whole $s$-plane.
	Furthermore, we have
	\begin{align*}
		\Lambda\lrb{F_a^{[\alpha,\beta]};s}&= i^a \Lambda\lrb{F_a^{[1-\beta,\alpha]};a-s}.
	\end{align*}
\end{theorem}
\begin{proof}
	Using Proposition~\ref{prop:Lamda-well-define-S_a}, \eqref{eq:F-a-converge-near-0}, \eqref{def:F-a-0-alpha-beta} and Theorem~\ref{thm:F-alpha-beta-transformation}~\ref{eq:F-alpha-beta-inversion}, a direct calculation gives
	\begin{multline}
		\Lambda\lrb{F_a^{[\alpha,\beta]};s}+\frac{1}{a}\!\lrb{\!\frac{B_a(\alpha)}{s} - \frac{i^{a}B_a(1-\beta)}{s-a} \!}\! \\
		  = \int_1^\infty F_{a,0}^{[\alpha,\beta]}(it)t^{s-1}dt + i^a\int_1^\infty F_{a,0}^{[1-\beta,\alpha]}\left(it\right)t^{a-s-1}dt .\label{eq:Lambda-function}
	\end{multline}
	By \eqref{def:F-a-alpha-beta} and \eqref{def:F-a-0-alpha-beta}, the integrands on the right-hand side have exponential decay towards infinity, giving the first part of the claim. The second claim follows by \eqref{def:F-a-0-alpha-beta} and \eqref{eq:Lambda-function}.
\end{proof}

By \eqref{def:Lambda-F-alpha-beta}, \eqref{def:F-a-0-alpha-beta}, \eqref{def:F-a-alpha-beta}, \eqref{Li} and \eqref{def:Hurwitz-zeta}, we may express $\smash{\Lambda(F_a^{[\alpha,\beta]};s)}$ explicitly in terms of polylogarithms and the Hurwitz zeta function.

\begin{lemma}\label{lem:Gamma-function-F-alpha-beta}
	For $\sigma>0$, we have
	\begin{align*}
	\Lambda\lrb{F_a^{[\alpha,\beta]};s}
	&= -\frac{\Gamma(s)}{(2\pi)^s} \lrb{(-1)^a\mathrm{Li}_s\lrb{e^{2\pi i\beta}}\zeta\lrb{s+1-a,1-\alpha} + \mathrm{Li}_s\lrb{e^{-2\pi i\beta}}\zeta\lrb{s+1-a,\alpha}} .
\end{align*}
\end{lemma}


Now we are ready to prove Theorem~\ref{thm1.3}.

\begin{proof}[Proof of Theorem~\ref{thm1.3}] 
Let  $a\ge2$ and $0<\alpha,\beta<1$. 
We have 
\begin{align*}
	&\widehat{\mathcal{E}}_{2-a}\lrb{-\frac{\alpha}{\tau}+\beta;-\frac{1}{\tau}} = (-\tau)^{2-a} \widehat{\mathcal{E}}_{2-a}\lrb{(1-\beta)\tau+\alpha;\tau}. \nonumber
\end{align*}
Since $0<\alpha,\beta<1$, using \eqref{def:widehatE-2-a-full} and Proposition~\ref{prop:non-holomorphic-part-as-integral-2}, the above can be rewritten as 
\begin{align}
	&\mathcal E_{2-a}\left(-\frac{\alpha}{\tau}+\beta;-\frac{1}{\tau}\right) - (-\tau)^{2-a} \mathcal E_{2-a}((1-\beta)\tau+\alpha;\tau) \nonumber\\
	&\hspace{1cm}= \frac{(4\pi v)^{a-1}}{a!\tau^{a-1}}\lrb{\frac{(-1)^{a+1} B_{a}\left(\alpha\right)}{\overline{\tau}^{a-1}} + B_{a}\left(1-\beta\right)\tau } + (-\tau)^{2-a}\mathcal{F}_{2-a}^{[1-\beta,\alpha]}(\tau)-\mathcal{F}_{2-a}^{[\alpha,\beta]}\left(-\frac{1}{\tau}\right).\label{eq:mathcalE-two-term-mid}
\end{align}
First, using Proposition~\ref{prop:non-holomorphic-part-as-integral-2}, we have 
\begin{align}
	& (-\tau)^{2-a}\mathcal{F}_{2-a}^{[1-\beta,\alpha]}(\tau) - \mathcal{F}_{2-a}^{[\alpha,\beta]}\left(-\frac{1}{\tau}\right) \label{eq:mathcalE-two-term-mid-part}\\
	&=\frac{(-2\pi i)^{a-1}}{(a-2)!}\sum_{n=0}^{a-2}\binom{a-2}{n}\tau^{n-a+2}\lrb{(-1)^{n}\int_{\frac{1}{\overline\tau}}^{i\infty} F_{a,0}^{[\a,\b]}(w) w^n dw - \int_{-\overline\tau}^{i\infty} F_{a,0}^{[1-\b,\a]}(w) w^{a-n-2} dw}.\nonumber
\end{align}
We next rewrite
\begin{align}
	& (-1)^{n} \int_{\frac{1}{\overline\tau}}^{i\infty} F_{a,0}^{[\a,\b]}(w) w^n dw - \int_{-\overline\tau}^{i\infty} F_{a,0}^{[1-\b,\a]}(w) w^{a-n-2} dw \nonumber\\
	&\hspace{3cm}= (-1)^{n} \int_{i}^{i\infty} F_{a,0}^{[\a,\b]}(w)w^n dw + (-1)^{n} \int_{\frac{1}{\overline{\tau}}}^{i} F_{a,0}^{[\a,\b]}(w) w^n dw \nonumber\\
	&\hspace{6cm}- \int_{i}^{i\infty} F_{a,0}^{[1-\b,\a]}(w) w^{a-n-2} dw  -\int_{-\overline{\tau}}^{i} F_{a,0}^{[1-\b,\a]}(w) w^{a-n-2} dw. \nonumber\\
	\intertext{Grouping together the first and third terms on the right-hand side and using \eqref{def:F-a-0-alpha-beta} and Theorem~\ref{thm:F-alpha-beta-transformation} \ref{eq:F-alpha-beta-inversion} in the second and fourth terms on the right-hand side gives}
	&(-1)^{n} \int_{\frac{1}{\overline\tau}}^{i\infty} F_{a,0}^{[\a,\b]}(w) w^n dw - \int_{-\overline\tau}^{i\infty} F_{a,0}^{[1-\b,\a]}(w) w^{a-n-2} dw \nonumber\\
	&\hspace{3cm}= i^{1-n} \lrb{ \int_{1}^{\infty} F_{a,0}^{[\a,\b]}(it) t^n dt + i^{a}\int_{1}^{\infty} F_{a,0}^{[1-\b,\a]}(it) t^{a-n-2} dt } \nonumber\\
	&\hspace{8.5cm}-\frac{1}{a}\int_{-\overline{\tau}}^{i}\lrb{\frac{B_a(\alpha)}{ w^a} - B_a(1-\b)} w^{a-n-2} dw.\nonumber
\end{align}

Finally, using Lemma~\ref{lem:Gamma-function-F-alpha-beta}  with $s=n+1$ gives 
\begin{align*}
	& (-1)^{n} \int_{\frac{1}{\overline\tau}}^{i\infty} F_{a,0}^{[\a,\b]}(w) w^n dw - \int_{-\overline\tau}^{i\infty} F_{a,0}^{[1-\b,\a]}(w) w^{a-n-2} dw \nonumber\\
	&\hspace{3.5cm}=i^{1-n} \Lambda\lrb{F_a^{[\alpha,\beta]};n+1}+\frac{(-1)^{n}}{a} \left(\frac{B_a(\alpha)}{n+1} \, \overline{\tau}^{-n-1} -\frac{ (-1)^aB_a(1-\beta)}{n+1-a} \, \overline{\tau}^{a-1-n}\right).
\end{align*}
Plugging this into \eqref{eq:mathcalE-two-term-mid-part} and then the resulting expression into \eqref{eq:mathcalE-two-term-mid} gives
\begin{align}
	&\mathcal E_{2-a}\left(-\frac{\alpha}{\tau}+\beta;-\frac{1}{\tau}\right) - (-\tau)^{2-a} \mathcal E_{2-a}((1-\beta)\tau+\alpha;\tau) \nonumber\\ &= \frac{(4\pi v)^{a-1}}{a!\tau^{a-1}}\lrb{\frac{(-1)^{a+1}B_{a}\left(\alpha\right)}{\overline{\tau}^{a-1}} + B_{a}\left(1-\beta\right)\tau}  {+} \frac{{ (-2\pi i)^{a-1}}}{(a-2)!}\sum_{n=0}^{a-2}\binom{a-2}{n}\tau^{n-a+2} \nonumber\\
	&\hspace{2.5cm} \times \vast( i^{1-n} \Lambda\lrb{F_a^{[\alpha,\beta]};n+1}+\frac{(-1)^{n}}{a} \left(\frac{B_a(\alpha)}{(n+1)\overline{\tau}^{n+1}}  -\frac{ (-1)^aB_a(1-\beta)}{(n+1-a)\overline{\tau}^{n+1-a}}\right)\vast).\label{eq:mathcalE-two-term-mid1}
\end{align}

We next evaluate $\Lambda(F_a^{[\alpha,\beta]};n+1)$ for $0\le n\le a-2$. 
By Lemma~\ref{lem:Gamma-function-F-alpha-beta} and \eqref{eq:sym-Ber}, we get
\begin{align*}
	\Lambda\lrb{F_a^{[\alpha,\beta]};n+1} &=\frac{(-1)^{n+1}n!}{(2\pi)^{n+1}}\lrb{ \mathrm{Li}_{n+1}\!\lrb{e^{ 2\pi i\beta}} + (-1)^{n+1} \mathrm{Li}_{n+1}\!\lrb{e^{- 2\pi i\beta}} }\frac{B_{a-n-1}(\alpha)}{a-n-1}\\
	&= i^{1-n}\frac{B_{n+1}(\beta)B_{a-n-1}(\alpha)}{(n+1)(a-n-1)},
\end{align*}
using \eqref{LiHur} and \eqref{eq:Hurwitz-zeta-Bernoulli}.
Substituting this into \eqref{eq:mathcalE-two-term-mid1} gives
\begin{align}
	&\mathcal E_{2-a}\left(-\frac{\alpha}{\tau}+\beta;-\frac{1}{\tau}\right) - {(-\tau)^{2-a}} \mathcal E_{2-a}((1-\beta)\tau+\alpha;\tau) \nonumber\\
	&={ \frac{(4\pi v)^{a-1}}{a!\tau^{a-1}}\lrb{\frac{(-1)^{a+1} B_{a}\left(\alpha\right)}{\overline{\tau}^{a-1}} + B_{a}\left(1-\beta\right)\tau}}  {-} \frac{{ (-2\pi i)^{a-1}}}{(a-2)!}\sum_{n=0}^{a-2}(-1)^{n}\binom{a-2}{n} \frac{B_{n+1}(\beta)B_{a-n-1}(\alpha)}{(n+1)(a-n-1)} \tau^{n-a+2} \nonumber\\
	&\hspace{2.5cm}{+} \frac{{ (-2\pi i)^{a-1}}}{a(a-2)!}\sum_{n=0}^{a-2}(-1)^{n}\binom{a-2}{n}\tau^{n-a+2} \left(\frac{B_a(\alpha)}{(n+1)\overline{\tau}^{n+1}}  -\frac{{ (-1)^a}B_a(1-\beta)}{(n+1-a)\overline{\tau}^{n+1-a}}\right).\label{eq:mathcalE-two-term-mid2}
\end{align}
We next simplify 
\begin{align*}
	\sum_{n=0}^{a-2} (-1)^{n}\binom{a-2}{n} \frac{\tau^{n-a+2}}{(n+1)\overline{\tau}^{n+1}} &=  -\frac{(-2iv)^{a-1}}{(a-1)\tau^{a-1}\overline{\tau}^{a-1}} + \frac{1}{(a-1)\tau^{a-1}},\\
	\sum_{n=0}^{a-2}(-1)^{n}\binom{a-2}{n}  \frac{ \tau^{n-a+2}}{(n+1-a)\overline{\tau}^{n+1-a}} &= \frac{(-2iv)^{a-1}}{(a-1)\tau^{a-2}} - \frac{(-1)^a\tau}{a-1}.
\end{align*}
Plugging these into \eqref{eq:mathcalE-two-term-mid2} and using \eqref{eq:sym-Ber} gives the theorem.
\end{proof}

\begin{proof}[Proof of Corollary~\ref{cor:thm1.3}] 
Abbreviating
$
	z_\tau^{[\varepsilon]}(x) := \frac{(1-2x) (1+\varepsilon\tau) + \tau +1}{2} 
$ with $\varepsilon\in\{\pm1\}$, \eqref{th1final} becomes
\begin{multline*}
	 \sum_{\varepsilon\in \{\pm1\}} \hspace{-0.2cm}\varepsilon\, \mathcal{E}_{2-2a}\left(z_\tau^{[\varepsilon]}(x);\tau\right)\\[-1em]
	 =-\tau^{2a-2} \sum_{\varepsilon\in \{\pm1\}} \hspace{-0.2cm}\varepsilon\, \mathcal{E}_{2-2a}\left(z_{-\frac{1}{\tau}}^{[\varepsilon]}(x);-\frac{1}{\tau}\right)+2(2\pi i)^{2a-1}\sum_{n=0}^{a-1}\frac{B_{2n+1}(x)B_{2a-1-2n}(x)}{(2n+1)!(2a-1-2n)!}\t^{2n}.
\end{multline*}
Corollary~\ref{cor:thm1.3} then follows using \eqref{def:P-ell+1}, \eqref{eq:sym-Ber} and taking the difference between $(\alpha,\beta)=(1-x,x)$ and $(\alpha,\beta)=(1-x,1-x)$ in Theorem~\ref{thm1.3} and Remark~\ref{rmk:Lim-general-ale1} with $a \mapsto 2a$.
\end{proof}

\section{Proof of Theorems~\ref{thm:torsion} and \ref{thm:D-a-1-widehatE-as-Eisenstein-series}}\label{sec:Proof of Theorems thm:torsion and thm:D-a-1-widehatE-as-Eisenstein-series}
In this section we connect $\widehat{\mathcal{E}}_{2-a}(z;\tau)$ at torsion points to (sesqui)harmonic Maass forms.
\begin{proof}[Proof of Theorem~\ref{thm:torsion}]\,\\
\ref{point:thm:torsion:1} and \ref{point:thm:torsion:2}
By Theorem~\ref{thm:mathcalE-2-a-Maass_Jacobi}, $\smash{\widehat{\mathcal{E}}_{2-a}(z;\tau)}$ is elliptic, and hence we may restrict to $0 \le \lambda, \mu < 1$. If $a = 2$, then we additionally assume that $(\lambda,\mu) \neq (0,0)$. Theorem~\ref{thm:mathcalE-2-a-Maass_Jacobi} \ref{point:thm:mathcalE-2-a-Maass_Jacobi:2},\ref{point:thm:mathcalE-2-a-Maass_Jacobi:1} imply that $\smash{\widehat{\mathcal{E}}_{2-a}(z;\tau)}$ transforms like a Jacobi form of weight $2-a$ and index $0$. Hence, by Lemma~\ref{lem:EichleZagier-torsion} with $m = 0$, $\smash{\widehat{\mathcal{E}}_{2-a}(\lambda\tau+\mu;\tau)}$ transforms modularly of weight $2-a$ on $\Gamma(N)$. 
Moreover, by \eqref{eq:asym-for-age3} and \eqref{eq:asym-for-a=2}, $\smash{\widehat{\mathcal{E}}_{2-a}(\lambda\tau+\mu;\tau) =O(v^{a-1})}$ as $v \to \infty$. By a straightforward computation using \eqref{def:widehatE-2-a-full}, $\smash{\Delta_{2-a}(\widehat{\mathcal{E}}_{2-a}(\lambda\tau+\mu;\tau))=0}$. Hence, $\smash{\widehat{\mathcal{E}}_{2-a}(\lambda\tau+\mu;\tau)}$ is a harmonic Maass form of weight $2-a$ on $\Gamma(N)$. \\
\ref{point:thm:torsion:3}
We first assume $z\neq0$. Using Theorem~\ref{thm:a=2case}, \eqref{eq:Kronecker-second-limit-formula}, \eqref{def:Dedekind-eta-function} and \eqref{def:Jacobi-theta-function}, we have, as $z\to0^+$, 
\begin{align*}
	\widehat{\mathcal{E}}_0(z;\tau) + 2\log|z| \sim   -2\log \left|\eta(\tau)\right|-2\log\lrb{2\pi} .
\end{align*}
Hence the function
\begin{align*}
	g(z;\tau) &:= \widehat{\mathcal{E}}_0(z;\tau) + 2\log|z|
\end{align*}
has no singularity at $z=0$. Define 
\begin{align}\label{def:g-tau-2}
	g(\tau):=g(0;\tau) = - 2\log\left|\eta(\tau)\right| - 2\log(2\pi) = \mathcal{E}_0(\tau)+\overline{\mathcal{E}_0(\tau)}  - 2\log(2\pi) + \frac{\pi v}6.
\end{align}
We next show that
\begin{align*}
	\widehat{g}(\tau)&:= g(\tau)-\log(v) = \lim_{z\to0^+}\left(\widehat{\mathcal E}_0(z;\tau) + 2\log|z|-\log(v)\right)
\end{align*}
is a sesquiharmonic Maass form of weight $0$. Using Theorem~\ref{thm:mathcalE-2-a-Maass_Jacobi} \ref{point:thm:mathcalE-2-a-Maass_Jacobi:1}, it is easy to check that $\widehat{g}$ is modular of weight $0$ on $\mathrm{SL}_2(\Z)$.
A direct calculation using \eqref{def:g-tau-2}, then shows that $\xi_0\circ\Delta_0\hspace{0.05cm}(\widehat{g}(\tau))=-\xi_0\circ\xi_2\circ\xi_0\hspace{0.05cm}(\widehat{g}(\tau))=0$. This gives part (3).
\end{proof}

\begin{proof}[Proof of Theorem~\ref{thm:D-a-1-widehatE-as-Eisenstein-series}]
By Theorem~\ref{thm:torsion}, $\widehat{\mathcal{E}}_{2-a}(\lambda\tau+\mu;\tau)$ is a weight $2-a$ harmonic Maass form. Hence, by Lemma~\ref{lem:action-od-D-k-on-HarmonicsMaassforms},	$\smash{D^{a-1}(\widehat{\mathcal{E}}_{2-a}(\lambda\tau+\mu;\tau))}$ is a modular form of weight $a$. Since, by Theorem~\ref{thm:mathcalE-2-a-Maass_Jacobi} \ref{point:thm:mathcalE-2-a-Maass_Jacobi:2},\ref{point:thm:mathcalE-2-a-Maass_Jacobi:1}, $\smash{\widehat{\mathcal{E}}_{2-a}(z;\tau)}$ is elliptic, we may assume $0\le\lambda,\mu<1$. Furthermore, by \cite[Theorem~5.5~(2)]{BFOR} and \eqref{def:widehatE-2-a-full}, we have
\begin{align}\label{eq:action-od-D-a-1-on-widehatmathcalE}
	D^{a-1}\lrb{\widehat{\mathcal{E}}_{2-a}(\lambda\tau+\mu;\tau)} &= D^{a-1}\lrb{\mathcal{E}_{2-a}(\lambda\tau+\mu;\tau)} - \frac{B_a(\lambda)}{a}.
\end{align} 

We next determine $\smash{D^{a-1}(\mathcal{E}_{2-a}(\lambda\tau+\mu;\tau))}$. By \eqref{def:P-ell+1}, for $(\lambda,\mu)\in\mathbb{Q}^2\cap [0,1)^2$ if $a\ge 3$ and $(\lambda,\mu)\in\mathbb{Q}^2\cap [0,1)^2\setminus\{(0,0\}$ if $a=2$, we write
\begin{align}
	\mathcal{E}_{2-a}(\lambda\tau+\mu;\tau) 
	&= \sum_{\substack{n\ge 1\\m\ge 0}} \frac{e^{2\pi in\mu}q^{(m+\lambda) n}}{n^{a-1}}{ +(-1)^a} \sum_{\substack{n\ge 1\\m\ge 0}} \frac{e^{-2\pi in\mu} q^{(m+1-\lambda)n}}{n^{a-1}}. \nonumber
\end{align}
Writing $\lambda=\tfrac{\lambda_1}{N}, \mu=\tfrac{\mu_1}{N}$ and letting $\z_N:=e^{\frac{2\pi i}{N}}$, this becomes
\begin{align*}
	\mathcal{E}_{2-a}(\lambda\tau+\mu;\tau) 
	&= \sum_{\substack{n\ge 1\\r\ge\lambda_1\\ r\equiv \lambda_1\Pmod{N}}}  \frac{\zeta_{N}^{\mu_1 n}q^{\frac{nr}{N}}}{n^{a-1}}{+(-1)^a} \sum_{\substack{n\ge 1\\r\ge N-\lambda_1\\ r\equiv -\lambda_1\Pmod{N}}} \frac{\zeta_{N}^{-\mu_1n}q^{\frac{nr}{N}}}{n^{a-1}}. \\
\intertext{Changing $n \mapsto -n$ and $r \mapsto -r$ in the second sum, and writing $r = \lambda_1 + kN$, we observe that $r \equiv \lambda_1 \pmod{N}$ implies $r \ge 1$ if and only if $r \ge \lambda_1$, and $r \le -1$ if and only if $r \le -N+ \lambda_1$. Thus,}
	\mathcal{E}_{2-a}(\lambda\tau+\mu;\tau)
	&= \sum_{\substack{n,r\ge 1\\ r\equiv \lambda_1 \pmod{N}}} \frac{\zeta_{N}^{\mu_1 n} q^{\frac{nr}{N}}}{n^{a-1}}
- \sum_{\substack{n,r\le -1\\ r\equiv \lambda_1 \pmod{N}}} \frac{\zeta_{N}^{\mu_1 n} q^{\frac{nr}{N}}}{n^{a-1}}.
\end{align*}
Applying $D^{a-1}$then  gives
\begin{align*}
	D^{a-1}(\mathcal E_{2-a}(\lambda\tau+\mu;\tau)) &= N^{1-a} \left(\sum_{\substack{n,r\ge1 \\ r\equiv\lambda_1\pmod{N}}} r^{a-1} \zeta_{N}^{\mu_1n} q^{\frac{nr}{N}} - \sum_{\substack{n,r\le-1 \\ r\equiv\lambda_1\pmod{N}}} r^{a-1} \zeta_{N}^{\mu_1n}q^{\frac{nr}{N}}\right).
\end{align*}
Using orthogonality of roots of unity and Theorem~\ref{thm:Fourier-exp-of-Eisenstein-ser}, we obtain
\begin{align}
	D^{a-1}(\mathcal E_{2-a}(\lambda\tau+\mu;\tau))
	&=  \frac{ (a-1)!}{(-2\pi i)^a} \sum_{\bm{m}\in\left(\Z/N\Z\right)^2} \zeta_{N}^{\mu_1m_1-\lambda_1m_2}G^{[N,\bm{m}]}_{a}(\tau) \nonumber\\ 
	&-\frac{ (a-1)!}{(-2\pi i)^a} \sum_{\bm{m}\in\left(\Z/N\Z\right)^2} \zeta_{N}^{\mu_1 m_1-\lambda_1m_2}\frac{\delta_{N|m_1}}{N^a}\lrb{\z\lrb{a,\frac{m_2}{N}}+(-1)^a\z\lrb{a,1-\frac{m_2}{N}}}\nonumber \\
	&= \frac{ (a-1)!}{(-2\pi i)^a} \sum_{\bm{m}\in\lrb{\Z/N\Z}^2} \zeta_{N}^{\mu_1 m_1-\lambda_1 m_2}G^{[N,\bm{m}]}_{a}(\tau) + \frac{B_a(\lambda)}{a},\nonumber
\end{align}
using \eqref{LiHur} and \eqref{eq:Hurwitz-zeta-Bernoulli}. Substituting this into \eqref{eq:action-od-D-a-1-on-widehatmathcalE} gives the theorem.
\end{proof}

\section{Concluding Remarks}\label{sec:Concluding Remarks}
We conclude by posing several questions for further exploration.
\begin{enumerate}[leftmargin=*]
	\item By Theorem~\ref{thm:mathbbE-as-mathcalE}, $\widehat{\mathcal{E}}_{2-a}(z;\tau)$ belongs to a larger family of so-called elliptic modular graph forms (see \cite[Section~1.2.1]{HiddingSchlottererVerbeek}). It would be interesting to determine whether these give rise to higher-dimensional analogues of harmonic Maass--Jacobi forms.
	\item Elliptic modular graph forms appear in the study of closed-string scattering amplitudes in string theory \cite{DHoker, DHokerGreenPioline}. A natural question is to what extent tools from the theory of harmonic Maass--Jacobi forms can be used to analyze these scattering amplitudes. A systematic study would be beneficial to both fields.
	\item There are several standard operators acting on Jacobi forms. It would be natural to study their action on $\smash{\widehat{\mathcal{E}}_{2-a}(z;\tau)}$, as well as on elliptic modular graph functions more generally. Along these lines, D'Hoker, Kleinschmidt, and Schlotterer \cite[Section~3]{DHokerKleinschmidtSchlotterer} studied the action of derivatives with respect to the $\tau$ and $z$, as well as the action of Laplacians on elliptic modular graph functions.
	\item Determining the generic shape of the Fourier expansion of a harmonic Maass--Jacobi form is a difficult problem. Under the additional assumption of holomorphicity in the $z$-variable, this was given in \cite[equation~(13)]{BR}. Although $\widehat{\mathcal{E}}_{2-a}(z;\tau)$ does not fall into this class, it nevertheless exhibits a ``nice'' Fourier expansion. This suggests extending the class of functions for which one can determine the Fourier expansion.
	\item For $a \le 1$, $\mathcal{E}_{2-a}(z;\tau)$ appears in the reduction formulas for $n$-point functions (see e.g. \cite[Sections~3--5]{Z96} and \cite[Sections~2 and~3]{BKT}). It would be interesting to investigate whether, for $a \ge 2$, $\mathcal{E}_{2-a}(z;\tau)$ and its completion also occurs somewhere in the theory of vertex operator algebras.
\end{enumerate}


\begin{thebibliography}{99}
   \bibitem{bln4}
G.~Andrews and B.~Berndt,
\emph{Ramanujan's lost notebook. Part IV},
Springer, New York, 2013.

\bibitem{BBG}
M.~Berg, K.~Bringmann, and T.~Gannon,
\emph{Massive deformations of Maass forms and Jacobi forms},
Commun. Number Theory Phys. \textbf{15} (2021), 575--603.

\bibitem{BS17}
B.~Berndt and A.~Straub,
\emph{Ramanujan’s formula for $\zeta(2m+1)$},
in \emph{Exploring the Riemann zeta function},
eds. H.~Montgomery, A.~Nikeghbali, and M.~Rassias,
pp.~13--34, Springer, Cham, 2017.

\bibitem{BFOR}
K.~Bringmann, A.~Folsom, K.~Ono, and L.~Rolen,
\emph{Harmonic Maass forms and mock modular forms: Theory and applications},
Amer. Math. Soc. Colloq. Publ., Vol.~64,
American Mathematical Society, Providence, RI, 2017.

\bibitem{BKT}
K.~Bringmann, M.~Krauel, and M.~Tuite,
\emph{Zhu reduction for Jacobi $n$-point functions and applications},
Trans. Amer. Math. Soc. \textbf{373} (2020), no.~5, 3261--3293.

\bibitem{BOW}
K.~Bringmann, K.~Ono, and I.~Wagner,
\emph{Eichler integrals of Eisenstein series as $q$-brackets of weighted $t$-hook functions on partitions},
Ramanujan J. \textbf{61} (2023), 279--293.


\bibitem{BR}
K.~Bringmann and O.~Richter,
\emph{Zagier-type dualities and lifting maps for harmonic Maass--Jacobi forms},
Adv. Math. \textbf{225} (2010), 2298--2315.

\bibitem{BZ23} F. Brunault and W. Zudilin, \emph{Modular regulators and multiple Eisenstein values}, arXiv:2303.15554.


\bibitem{CS2017}
H.~Cohen and F.~Str\"omberg,
\emph{Modular forms: A classical approach},
Graduate Studies in Mathematics, Vol.~179,
American Mathematical Society, Providence, RI, 2017.

\bibitem{DHoker}
E.~D'Hoker, M.~Green, \"O.~G\"urdo\u{g}an, and P.~Vanhove,
\emph{Modular graph functions},
Commun. Number Theory Phys. \textbf{11} (2017), no.~1, 165--218.

\bibitem{DHokerGreenPioline}
E.~D'Hoker, M.~Green, and B.~Pioline,
\emph{Asymptotics of the $D^8R^4$ genus-two string invariant},
Commun. Number Theory Phys. \textbf{13} (2019), no.~2, 351--462.

\bibitem{DHokerKleinschmidtSchlotterer}
E.~D'Hoker, A.~Kleinschmidt, and O.~Schlotterer,
\emph{Elliptic modular graph forms I: Identities and generating series},
J. High Energy Phys. \textbf{2021} (2021), 3, 151.

\bibitem{DS2005}
F.~Diamond and J.~Shurman,
\emph{A first course in modular forms},
Graduate Texts in Mathematics, Vol.~228,
Springer, New York, 2005.

\bibitem{EZ1985}
M.~Eichler and D.~Zagier,
\emph{The theory of Jacobi forms},
Progr. Math., Vol.~55,
Birkh\"auser, Boston, 1985.

\bibitem{HiddingSchlottererVerbeek}
M.~Hidding, O.~Schlotterer, and B.~Verbeek,
\emph{Elliptic modular graph forms II: Iterated integrals},
arXiv.2208.11116.

\bibitem{lim}
S.~Lim,
\emph{A class of infinite series from generalized Eisenstein series},
Honam Math. J. \textbf{34} (2012), no. 3, 391--402.

\bibitem{L}
A.~Libgober,
\emph{Elliptic genera, real algebraic varieties and quasi-Jacobi forms},
in \emph{Topology of Stratified Spaces},
Math. Sci. Res. Inst. Publ., {\bf58},
Cambridge Univ. Press, Cambridge, 2011, pp.~95--120.

\bibitem{Nist}
F.~Olver, D.~Lozier, R.~Boisvert, and C.~Clark,
\emph{NIST handbook of mathematical functions},
Cambridge University Press, Cambridge, 2010.

\bibitem{PW2001}
P.~Pasles and W.~Pribitkin,
\emph{A generalization of the Lipschitz summation formula and some applications},
Proc. Amer. Math. Soc. \textbf{129} (2001), no.~11, 3177--3184.

\bibitem{R88}
S.~Ramanujan,
\emph{The lost notebook and other unpublished papers},
Narosa, New Delhi, 1988.

\bibitem{SchlottererSohnleTao}
O.~Schlotterer, Y.~Sohnle, and Y.-X.~Tao,
\emph{Elliptic modular graph forms, equivariant iterated integrals and single-valued elliptic polylogarithms},
arXiv:2511.15883.

\bibitem{SI}
C.~Siegel,
\emph{Lectures on advanced analytic number theory},
Tata Institute of Fundamental Research, Bombay, 1961.

\bibitem{Zag}
D.~Zagier,
\emph{The Bloch--Wigner--Ramakrishnan polylogarithm function},
Math. Ann. \textbf{286} (1990), 613--624.

\bibitem{Z96}
Y.~Zhu,
\emph{Modular invariance of characters of vertex operator algebras},
J. Amer. Math. Soc. \textbf{9} (1996), no.~1, 237--302.

\bibitem{ZW}
S.~Zwegers,
\emph{Mock theta functions},
Ph.D. thesis,
Universiteit Utrecht, 2002.
\end{thebibliography}
\end{document}